\documentclass[12pt]{extarticle}
\usepackage{extsizes}
\usepackage[reqno]{amsmath} % equation numbers on the right
\usepackage{amsmath}
\usepackage{amssymb}
\usepackage{amsthm}
\usepackage{amscd}
\usepackage{amsfonts}
\usepackage{graphicx}
\usepackage{dsfont}
\usepackage{fancyhdr}
\usepackage{enumerate}
\usepackage{youngtab}
\usepackage[utf8]{inputenc}
\usepackage{comment}
\usepackage{mathtools}
\usepackage{subfigure}
\usepackage[margin=1.5in]{geometry}
\usepackage{graphicx}
\usepackage{algorithm}
\usepackage{algpseudocode}
\usepackage{color}
\usepackage[hidelinks]{hyperref}
\usepackage{notation}
\usepackage{tikz}
\usetikzlibrary{positioning}

\newtheorem{corollary}{Corollary}
\newtheorem{definition}[corollary]{Definition}
\newtheorem{lemma}[corollary]{Lemma}
\newtheorem{lemma*}[lem6]{Lemma}

\newtheorem{proposition}[corollary]{Proposition}
\newtheorem{theorem}[corollary]{Theorem}
\newtheorem{example}[corollary]{Example}
\newtheorem{question}[corollary]{Question}

\def\uni{\mathrm{uni}}
\def\ab{\mathrm{ab}}

\newcommand{\defeq}{\vcentcolon=}

\begin{document}

\AtEndDocument{%
\par
\medskip
\begin{tabular}{@{}l@{}}%
\textsc{Thomás Jung Spier}\\
\textsc{Dept. of Combinatorics and Optimization}\\ 
\textsc{University of Waterloo, Canada}\\
  \emph{E-mail address}:\texttt{tjungspier@uwaterloo.ca}
\end{tabular}}

\title{Eigenvalues of Universal Covers and the Matching Polynomial}
\author{Thomás Jung Spier
\footnote{tjungspier@uwaterloo.ca}}
\date{\today}
\maketitle
\author
\vspace{-0.8cm}

\begin{abstract} 
    In this work, we prove that the universal and maximal abelian covers of a finite multi-graph have the same eigenvalues. This result strengthens a recent theorem of Li, Magee, Sabri, and Thomas (2025) and answers one of their questions. Our proof builds upon their new characterization of the point spectrum of maximal abelian covers in terms of matching polynomials. It is based on the theory of the matching polynomial and its Gallai–Edmonds decomposition.
\end{abstract}

\begin{center}
\end{center}

%%%%%%%%%%%%%%%%%%%%%%%%%%%%%%%%%%%%%%%%%%%%%%%%%%%%%%%%%%%%%%%%%%%%%%%%%%%%%%%%

\section{Introduction}\label{sec:introduction}\

In this article we study the spectrum of the universal and maximal abelian covers of a multi-graph. The spectral theory of these covers has connections with  many areas, among those mathematical physics, analysis, probability and combinatorics. See~\cite{aomoto1991point,avni2020periodic,avni2022periodic,banks2024useful,banks2022point,bordenave2019eigenvalues,christiansen2021remarks,figa1994harmonic,garza2023spectra,li2025eigenvalues, sabri2023flat} for some recent developments in this area.

We next provide an overview of some known results together with our main contributions. Precise definitions and statements are deferred to Sections~\ref{sec:preliminaries} and~\ref{sec:main_results}.

Let $G = (V(G), E(G))$ be a finite weighted multi-graph. Each vertex $i\in V(G)$ is assigned a weight $r_i\in\mathbb{R}$, and each arc $e\in\vec{E}(G)$ is assigned a weight $\rho_e\in\mathbb{C}$, with $\rho_{e^{-1}} =\overline{\rho_e}\neq 0$, where $e^{-1}$ denotes the reverse of $e$. For simplicity, we will refer to such an object as a multi-graph. For every multi-graph $G$, one can define its universal and maximal abelian covers, denoted $G^\uni$ and $G^\ab$, respectively, together with their adjacency operators, spectra, and eigenvalues.

Recently, Banks, Garza-Vargas, and Mukherjee~\cite{banks2022point}, building on the work of Aomoto~\cite{aomoto1988algebraic, aomoto1991point}, established a complete characterization of the eigenvalues of the universal cover $G^\uni$ in terms of the structure of the base multi-graph $G$.

\begin{theorem}[{\cite[Cor. 3.4]{banks2022point}}]\label{thm:universal_cover_eigenvalue} A real number $\theta$ is an eigenvalue of $G^\uni$ if and only if $G$ has a $\theta$--Aomoto subset.
\end{theorem}

An immediate consequence of this result is that every eigenvalue of the universal cover $G^\uni$ is also an eigenvalue of $G$. Moreover, as noted in~\cite[Cor. 3.4]{banks2022point}, Theorem~\ref{thm:universal_cover_eigenvalue} implies that there is a finite-time algorithm to compute the eigenvalues of $G^\uni$.
 
As observed in~\cite[Sec. 1.1]{banks2022point}, Aomoto~\cite{aomoto1991point} used Theorem~\ref{thm:universal_cover_eigenvalue} to establish the following result. For this statement, we say that a multi-graph $G$ is \emph{regular} if every vertex has the same number of incident edges, counting self-loops twice.  

\begin{theorem}[{\cite[Thm. 2]{aomoto1991point}}]\label{thm:universal_cover_regular} If $G$ is regular, then $G^\uni$ has no eigenvalues. 
\end{theorem}

Recently, Li, Magee, Sabri, and Thomas~\cite{li2025eigenvalues} established analogues of Theorems~\ref{thm:universal_cover_eigenvalue} and~\ref{thm:universal_cover_regular} for the maximal abelian cover $G^\ab$ of a multi-graph $G$. To state their results, if $H$ is a subgraph of $G$, we write $G\setminus H$ for the subgraph of $G$ induced by $V(G)\setminus V(H)$.

\begin{theorem}[{\cite[Thm. 1.1 and Rmk. 3.10]{li2025eigenvalues}}]\label{thm:abelian_cover_eigenvalue} A real number $\theta$ is an eigenvalue of $G^\ab$ if and only if it is a zero of the matching polynomial of $G\setminus\Gamma$ for every $2$--regular subgraph $\Gamma$ of $G$. Moreover, the same statement holds if the matching polynomial is replaced by the characteristic polynomial.  
\end{theorem}

\begin{theorem}[{\cite[Thm. 1.2]{li2025eigenvalues}}]\label{thm:abelian_cover_regular} If $G$ is regular, then $G^\ab$ has no 
eigenvalues.   
\end{theorem}

Theorems~\ref{thm:abelian_cover_eigenvalue} and~\ref{thm:abelian_cover_regular} answered Problem~6.11 and Conjecture~6.12 of Higuchi and Nomura~\cite{higuchi2009spectral}, respectively. In the course of establishing these results, Li et al.~\cite{li2025eigenvalues} also proved the following theorem, which they also claim follows from work of Arizmendi, Cébron, Speicher, and Yin~\cite{arizmendi2024universality}.

\begin{theorem}[{\cite[Prop. A.2]{li2025eigenvalues}}]\label{thm:universal_cover_implies_abelian_cover} Every eigenvalue of $G^\uni$ is an eigenvalue of $G^\ab$.
\end{theorem}

These results above suggest that the eigenvalues of the universal and maximal abelian cover of a multi-graph always coincide. In this direction, Li et al.~\cite[p. 3]{li2025eigenvalues} pose the following question (specifically for multi-graphs with vertex weights $0$ and arc weights $1$).

\begin{question}[{\cite[p. 3]{li2025eigenvalues}}]\label{question:main} Is every eigenvalue of $G^\ab$ an eigenvalue of $G^\uni$?
\end{question}

Moreover Li et al.~\cite[p. 3]{li2025eigenvalues} write: \emph{``It would be highly desirable to resolve this matter one way or the other.''} 

In this work, our main result is an affirmative solution to Question~\ref{question:main} obtaining the following statement.

\begin{theorem}\label{thm:main_theorem_introduction_1} $G^\uni$ and $G^\ab$ have the same eigenvalues.
\end{theorem}

By combining Theorems~\ref{thm:universal_cover_regular} and~\ref{thm:main_theorem_introduction_1}, we also obtain an alternative proof of Theorem~\ref{thm:abelian_cover_regular}, thereby providing an alternative solution to Conjecture 6.12 of Higuchi and Nomura~\cite{higuchi2009spectral}. We refer the reader to Section~\ref{sec:main_results} for the full characterization of the eigenvalues of $G^\ab$ and $G^\uni$.

Theorem~\ref{thm:main_theorem_introduction_1} is a consequence of the following result together with Theorems~\ref{thm:universal_cover_eigenvalue},\ref{thm:abelian_cover_eigenvalue}, and~\ref{thm:universal_cover_implies_abelian_cover}.  

\begin{theorem}\label{thm:main_theorem_introduction_2} Let $G$ be a multi-graph without a $\theta$--Aomoto subset. Then there exists a $2$-regular subgraph $\Gamma$ of $G$ such that $\theta$ is not a zero of the matching polynomial of $G\setminus\Gamma$.
\end{theorem}

To prove Theorem~\ref{thm:main_theorem_introduction_2}, we use techniques originally developed in the study of the Gallai--Edmonds decomposition for the matching polynomial. This decomposition was introduced by Godsil~\cite{godsil1995algebraic} and further developed by Ku and Chen~\cite{ku2010analogue}, Ku and Wong~\cite{ku2009maximum,ku2010extensions, ku2013gallai}, Bencs and M{\'e}sz{\'a}ros~\cite{bencs2022atoms}, and the author~\cite{spier2021graph, spier2023refined}. We work with the matching polynomial rather than the characteristic polynomial precisely because this decomposition is available for it.

It would be interesting to obtain a proof of Theorem~\ref{thm:main_theorem_introduction_1} without using the Gallai--Edmonds decomposition for the matching polynomial. One could attempt this by working directly with covers or with the characteristic polynomial instead.

The remainder of the paper is organized as follows. In Section~\ref{sec:preliminaries}, we set up the definitions and present previous work on this subject. In Section~\ref{sec:main_results}, we state Theorems~\ref{thm:main_theorem_introduction_1} and~\ref{thm:main_theorem_introduction_2} in their final form. In Section~\ref{sec:new_gallai_edmonds}, we present new results about the Gallai--Edmonds decomposition for the matching polynomial. In Section~\ref{sec:proof_of_main_results}, we complete the proof of Theorem~\ref{thm:main_theorem_introduction_2}.

%%%%%%%%%%%%%%%%%%%%%%%%%%%%%%%%%%%%%%%%%%%%%%%%%%%%%%%%%%%%%%%%%%%%%%%%%%%%%%%%

\section{Preliminaries}\label{sec:preliminaries}\

In this section, and in the remainder of the paper, we let $G = (V(G), E(G))$ be a finite weighted multi-graph. For an arc $e\in \vec{E}(G)$, its origin and terminus are denoted as $o(e)$ and $t(e)$,
respectively. Each vertex $i\in V(G)$ is assigned a weight $r_i\in\mathbb{R}$, and each arc $e\in\vec{E}(G)$ is assigned a weight $\rho_e\in\mathbb{C}$, with $\rho_{e^{-1}} =\overline{\rho_e}\neq 0$, where $e^{-1}$ denotes the reverse of $e$. For simplicity, we will refer to such an object as a multi-graph.

In this case, the corresponding \emph{adjacency matrix} of $G$, denoted by $A^G$, is the Hermitian matrix defined by  
\[
A^G_{ij}\defeq r_i \delta_{ij} + \sum_{\substack{e \in \vec{E}(G)\\ o(e)=i, t(e)=j}} \rho_e, \qquad \text{for all } i,j \in V(G).
\]
Note that $A^G_{ij}=0$ whenever $i\neq j$ and $i\nsim j$.

We define the \emph{characteristic polynomial} of $G$ to be
\[
\phi^G(x)\defeq\det(xI - A^G).
\]

A multi-graph without self-loops or multiple edges is called a graph. Note that subgraphs of a multi-graph are not required to be graphs and may themselves be multi-graphs. If $S$ is a subset of the vertices of the multi-graph $G$, we write $G\setminus S$ for the subgraph of $G$ induced by $V(G)\setminus S$, and $G[S]$ for the subgraph induced by $S$. The \emph{frontier} of $S$ in $G$, denoted by $\partial S=\partial_G S$, is the set of vertices outside $S$ that have a neighbor in $S$. Analogously, if $H$ is a (directed) subgraph of $G$, we write $G\setminus H$ for $G\setminus V(H)$, define $\partial H$ as $\partial V(H)$, and let $\mathrm{cc}(H)$ denote the number of components of $H$. A multi-graph is called \emph{regular} if every vertex has the same number of incident edges, counting self-loops twice.

Note that two distinct parallel edges of $G$ form a cycle of length $2$, and a loop forms a cycle of length $1$. A $2$--regular subgraph of $G$ is a disjoint union of cycles of $G$, which may include cycles of length $1$ or $2$. A \emph{matching} of the multi-graph $G$ is a subset of edges forming a $1$--regular subgraph of $G$. A multi-graph without cycles is called a \emph{forest}. In particular, a forest has no self-loops or multiple edges and is therefore a graph. A \emph{tree} is a connected forest.

%%%%%%%%%%%%%%%%%%%%%%%%%%%%%%%%%%%%%%%%%%%%%%%%%

\subsection{Matching polynomial}\label{sec:matching_pol}\

We begin with an overview of the theory of the matching polynomial and its Gallai–Edmonds decomposition, as developed in~\cite{spier2021graph, spier2023refined}. Although this theory was originally developed for graphs rather than multi-graphs, the results extend directly to multi-graphs with virtually no modifications. For this reason, we state the results for multi-graphs, indicating explicitly when the graph assumption is needed.

Let $\mathcal{M}_G$ be the set of all matchings of $G$. For simplicity, we write $i\notin M$ if the vertex $i$ is not saturated by the matching $M$. For an edge $e\in E(G)$, we set $\lambda_e\defeq -\lvert\rho_{\vec{e}}\rvert^2$, where $\vec{e}$ is any of the arcs corresponding to the edge $e$.

The \emph{matching polynomial} of $G$ is then defined as  
\[
\mu^G(x)\defeq\sum_{M\in\mathcal{M}_G}\prod_{i\notin M} (x-r_i)\prod_{e\in M}\lambda_e.
\]
Note that the loop weights of $G$ do not contribute to the matching polynomial of $G$.

The following result, due to Godsil and Gutman~\cite{godsil1978matching}, establishes a connection between the matching and characteristic polynomials of a graph. We note that this result does not hold in general for multi-graphs that are not graphs.

\begin{theorem}[{\cite[Cor. 2.1]{godsil1978matching}}]\label{thm:matching_equals_characteristic} Let $G$ be a graph. Then the matching and characteristic polynomials of $G$ are equal if and only if $G$ is a forest.
\end{theorem} 

For each vertex $i$ of $G$, we also define the \emph{graph continued fraction} of $G$ rooted at $i$ as  
\[
\alpha_i^G(x)\defeq\dfrac{\mu^G}{\mu^{G\setminus i}}(x).
\]

This definition is motivated by a result of Viennot~\cite[p. 149]{viennot1985combinatorial}, which establishes a connection between matching polynomials and branched continued fractions. The following two lemmas illustrate part of this connection. For further details, we refer the reader to~\cite{spier2023refined}. 

To state the next result, we denote by $E_G(i,j)$ the set of edges connecting the vertices $i$ and $j$ in the multi-graph $G$.

\begin{lemma}[{\cite[p. 4]{spier2023refined}}]\label{lem:alpha_recursion} 
Let $i$ be a vertex of $G$. Then
\[
\alpha_i^G(x)=x-r_i+\sum_{j\sim i}\sum_{e\in E_G(i, j)}\dfrac{\lambda_e}{\alpha_j^{G\setminus i}(x)}.  
\]  
\end{lemma}

For a rooted multi-graph $G$ with root $i$, the \emph{rooted path tree} $T_i^G$ is the rooted tree whose vertices correspond to paths in $G$ starting at $i$, where two vertices are adjacent if one path is a maximal sub-path of the other. The root of $T_i^G$ is the trivial path consisting of the single vertex $i$.

The weights of $T_i^G$ are inherited from the weights of $G$ as follows:
\begin{itemize}
 \item Each vertex of $T_i^G$ corresponding to a path ending at $j$ is assigned the vertex weight $r_j$;
 \item If $P : i\to j$ is a maximal sub-path of $\hat{P} : i\to k$, then the edge $P\hat{P}$ of $T_i^G$ connecting the vertices corresponding to $P$ and $\hat{P}$ has weight $\lambda_{P\hat{P}} :=\lambda_{jk}$. 
\end{itemize}

In this context, we have the following fundamental result of Godsil~\cite{godsil_matchings}, which establishes a connection between the graph continued fraction of $G$ at a vertex $i$ and that of its rooted path tree $T_i^G$.  

\begin{lemma}[{\cite[Thm. 2.5]{godsil_matchings}}]\label{lem:path_tree}  
Let $G$ be a multi-graph and $i$ a vertex of $G$. Then  
\[
\alpha_i^G(x)=\alpha_i^{T_i^G}(x).  
\]  
\end{lemma}

Using Lemmas~\ref{lem:alpha_recursion} and~\ref{lem:path_tree}, one can deduce the following classical result of Heilmann and Lieb~\cite[Thm. 4.2]{heilmann-lieb}. For this result, and throughout the paper, we denote by $m_\theta(G)$ the multiplicity of $\theta$ as a zero of $\mu^G(x)$.

\begin{theorem}[{\cite[Thm. 4.2]{heilmann-lieb}}]\label{thm:Heilmann-Lieb} Let $G$ be a multi-graph and $i$ one of its vertices. Then all zeros of $\mu^G(x)$ are real. Moreover, the polynomials $\mu^G(x)$ and $\mu^{G\setminus i}(x)$ interlace, that is, between any two zeros of $\mu^G(x)$ there is a zero of $\mu^{G\setminus i}(x)$, and vice versa. In particular, for every real number $\theta$,  
\[
m_\theta(G\setminus i)\in\{\, m_\theta(G),\; m_\theta(G)\pm 1\,\}.
\]  
\end{theorem}

It is in fact possible to provide a more precise location for the zeros of $\mu^G(x)$ in terms of the weights of $G$ (see, for example,~\cite[Cor. 10]{spier2021graph}). Such estimates were originally established by Heilmann and Lieb~\cite{heilmann-lieb} and played a key role in the construction of bipartite Ramanujan graphs of all degrees in~\cite{interlacing_familiesI}.

Using Theorem~\ref{thm:Heilmann-Lieb}, one can also show (see~\cite[Cor. 12]{spier2023refined}) that all zeros and poles of $\alpha_i^G(x)$ are simple, real, and interlace. Moreover, the function $\alpha_i^G:\Rds\to\Rds$ is increasing and bijective on each of its branches.

We now turn to the $\theta$--Gallai-Edmonds decomposition of $G$, defined for each real number $\theta$. The vertices of $G$ are partitioned into three sets according to the value of the graph continued fraction $\alpha_i^G(\theta)$:
\[
i\in 0_\theta^G\,\text{if }\alpha_i^G(\theta)=0,\quad  
i\in\infty_\theta^G\,\text{if }\alpha_i^G(\theta)=\infty,\quad  
i\in\pm_\theta^G\,\text{otherwise}.
\]  
This yields the partition $V(G)=\pm_\theta^G\sqcup 0_\theta^G\sqcup\infty_\theta^G$. By Theorem~\ref{thm:Heilmann-Lieb}:  
\begin{equation}\label{eq:vertex_classification}
\begin{aligned}
& i\in 0_\theta^G &&\Leftrightarrow && m_\theta(G\setminus i) = m_\theta(G) - 1,\\[6pt]
& i\in\pm_\theta^G &&\Leftrightarrow && m_\theta(G\setminus i) = m_\theta(G),\\[6pt]
& i\in\infty_\theta^G &&\Leftrightarrow && m_\theta(G\setminus i) = m_\theta(G) + 1.
\end{aligned}
\end{equation}

Godsil~\cite{godsil1995algebraic} observed that when the multi-graph is a graph, the vertex and arc weights are $0$ and $1$, respectively, and $\theta = 0$, the $\theta$--Gallai--Edmonds decomposition reduces to the classical Gallai--Edmonds decomposition, as presented in~\cite{edmonds1965paths, gallai1963kritische, lovasz1986matching}. The $\theta$--Gallai--Edmonds decomposition for forests has also been extensively studied under the name \emph{Parter--Wiener theory}, as detailed by Johnson and Saiago~\cite{johnson2018eigenvalues}.

One of the motivations for introducing the $\theta$--Gallai--Edmonds decomposition was to prove that the matching polynomials of vertex-transitive graphs (with vertex and arc weights $0$ and $1$, respectively) have simple zeros. This was accomplished by Ku and Chen in~\cite{ku2010analogue}.

We now describe some results obtained by Ku and Chen in~\cite{ku2010analogue}, and later refined by Ku and Wong in~\cite{ku2013gallai}. Following Godsil~\cite{godsil1995algebraic}, a multi-graph $G$ is said to be \emph{$\theta$--critical} if $V(G) = 0_\theta^G$. The $\theta$--critical components of $G$ are the components of the subgraph induced by $0_\theta^G$. 

In this context, we recall the following result, known as Gallai’s Lemma~\cite[Thm. 1.7]{ku2010analogue}, proved by Ku and Wong~\cite{ku2013gallai}.

\begin{theorem}[{\cite[Thm. 4.13]{ku2013gallai}}]\label{thm:gallais_lemma}  
If $G$ is a connected $\theta$--critical multi-graph, then $m_\theta(G)=1$.  
\end{theorem}  

It was shown by Ku and Wong~\cite[Lem. 4.1]{ku2013gallai} (or~\cite[Prop. 26]{spier2021graph}) that $\partial 0_\theta^G\subseteq\infty_\theta^G$. The two sets $0_\theta^G$ and $\partial 0_\theta^G$ play a central role in the $\theta$--Gallai--Edmonds decomposition. In particular, these sets determine the multiplicity $m_\theta(G)$, as proved by Ku and Wong~\cite[Cor. 4.14]{ku2013gallai} (or~\cite[Cor. 31]{spier2021graph}), and are related as described in~\cite[Cor. 32]{spier2023refined}.  

\begin{theorem}[{\cite[Cor. 4.14]{ku2013gallai}}]\label{thm:multiplicity}  
The multiplicity of $\theta$ as a zero of $\mu^G$ is equal to the number of $\theta$--critical components of $G$ minus the number of vertices in $\partial 0_\theta^G$, that is, $m_\theta(G)=\mathrm{cc}(G[0^G_\theta])-\lvert\partial 0^G_\theta\rvert$.  
\end{theorem}  

\begin{theorem}[{\cite[Cor. 32]{spier2023refined}}]\label{thm:matched_special}  
Every nonempty subset $U\subseteq\partial 0_\theta^G$ is adjacent to least $\lvert U\rvert + 1$ $\theta$--critical components of $G$.  
\end{theorem}  

We also note that Bencs and M{\'e}sz{\'a}ros~\cite[Lem. 1.9 and Thm.1.11]{bencs2022atoms} established versions of Theorems~\ref{thm:gallais_lemma} and~\ref{thm:matched_special} for infinite graphs with vertex and arc weights $0$ and $1$, respectively.

The key result used to prove Theorems~\ref{thm:gallais_lemma},~\ref{thm:multiplicity}, and~\ref{thm:matched_special} is the following stability result, presented in its final form in~\cite{spier2023refined}.

\begin{theorem}[{\cite[Thm. 29]{spier2023refined}}]\label{thm:stability} If $i\in\partial 0_\theta^G$, then $\alpha_j^{G\setminus i}(\theta) =\alpha_j^G(\theta)$ for every vertex $j\neq i$.  
\end{theorem}  

Note that Theorem~\ref{thm:stability} provides a satisfactory description of how the $\theta$--Gallai-Edmonds decomposition changes when a vertex in $\partial 0_\theta^G$ is deleted.

The next lemma is one of the main tools used in~\cite{spier2023refined} to prove Theorem~\ref{thm:stability}. In order to state this result, for two vertices $i$ and $j$ of $G$ we denote by $[i\to j]$ the set of paths starting at $i$ and ending at $j$. We also write $\lambda_P$ for the product of $-\lambda_e$ over the edges of the path $P$.  

\begin{lemma}[{\cite[Lem. 9]{spier2023refined}}]\label{lem:contraction} Given a multi-graph $G$ and two distinct vertices $i$ and $j$, we have  
\[
\alpha_i^G(x) =\alpha_i^{G\setminus j}(x) +\frac{\lambda_{i\sim j}^G(x)}{\alpha_j^{G\setminus i}(x)},  
\]  
\noindent where  
\[
\lambda_{i\sim j}^G(x) =\lambda_{j\sim i}^G(x) := -\sum_{P\in [i\to j]}\lambda_P\left(\frac{\mu^{G\setminus P}}{\mu^{G\setminus\{i,j\}}}(x)\right)^2.  
\]  
\end{lemma}

For precise statements on the consequences of Lemma~\ref{lem:contraction}, see~\cite[Prop. 20--24 and Fig. 5]{spier2023refined}. We also note that an analogue of Lemma~\ref{lem:contraction} for the characteristic polynomial has recently been used in~\cite{coutinho2023strong}. 

To prove Theorem~\ref{thm:main_theorem_introduction_2}, we also need to analyze how deleting vertices in $0_\theta^G$, as well as paths and cycles of $G$, affects the $\theta$--Gallai--Edmonds decomposition. For this reason, we now study the effect of removing a path or cycle from $G$ on $m_\theta(G)$, the multiplicity of $\theta$ as a zero of $\mu^G(x)$.

We recall the following result of Godsil~\cite[Cor. 2.5]{godsil1995algebraic} (see also~\cite[Lem. 16]{spier2023refined}), along with a simple consequence.  

\begin{lemma}[{\cite[Cor. 2.5]{godsil1995algebraic} }]\label{lem:critical_path} Let $P: i\to j$ be a path in $G$. Then $m_\theta(G\setminus P)\geq m_\theta(G) - 1$, with equality only if both $i$ and $j$ are in $0_\theta^G$.  
\end{lemma}  

\begin{corollary}\label{cor:critical_cycle} Let $C$ be a cycle in $G$. Then $m_\theta(G\setminus C)\geq m_\theta(G) - 1$,  with equality only if $C$ is contained in a $\theta$--critical component of $G$. 
\end{corollary}  
\begin{proof}  Every vertex of $C$ is the starting point of a path in $G$ that uses the same  vertices as $C$. The claim follows from Lemma~\ref{lem:critical_path}.  
\end{proof}  

We call a path or cycle that attains equality in Lemma~\ref{lem:critical_path} or Corollary~\ref{cor:critical_cycle} a \emph{$\theta$--critical path} or \emph{$\theta$--critical cycle} of $G$, respectively. 

The next result establishes the existence of $\theta$--critical paths under certain circumstances. While it can be proved using the same strategy as in~\cite[Lem. 3.8]{godsil1995algebraic}, we present here an alternative proof based on the results of~\cite{spier2023refined}.

\begin{lemma}[{\cite[Lem. 3.8]{godsil1995algebraic}}]\label{lem:mult1_implies_critical_path} Assume that $m_\theta(G) = 1$, and let $i, j\in 0^G_\theta$. Then there exists a $\theta$--critical path $P:i\to j$ in $G$.
\end{lemma}
\begin{proof}  
Since $i\in 0^G_\theta$ and $m_\theta(G) = 1$, it follows that $m_\theta(G\setminus i) = 0$. As a consequence, $0^{G\setminus i}_\theta =\emptyset$ and $\alpha^{G\setminus i}_j(\theta)\neq 0$. By Lemma~\ref{lem:contraction} (or~\cite[Prop. 21]{spier2023refined}), this implies $\lambda_{i\sim j}^G(\theta)\neq 0$.  

Recall that  
\[
\lambda_{i\sim j}^G(\theta) = -\sum_{P\in [i\to j]}\lambda_P\left(\dfrac{\mu^{G\setminus P}}{\mu^{G\setminus\{i,j\}}}(\theta)\right)^2.
\]  

If $\lambda_{i\sim j}^G(\theta)\in (-\infty,0)$, then by~\cite[Prop. 22(f) or Fig. 5]{spier2023refined} we have $j\notin\infty^{G\setminus i}_\theta$ and $m_\theta(G\setminus\{i,j\}) = 0$. Since $\lambda_{i\sim j}^G(\theta)\in (-\infty,0)$, by the expression above, there exists a path $P\in [i\to j]$ with  
\[
m_\theta(G\setminus P)\leq m_\theta(G\setminus\{i,j\}) = 0,
\]  
\noindent so $m_\theta(G\setminus P) = 0$, and $P$ is a $\theta$--critical path in $G$.  

On the other hand, if $\lambda_{i\sim j}^G(\theta) = -\infty$, then by~\cite[Prop. 20 or Fig. 5]{spier2023refined} we have $j\in\infty^{G\setminus i}_\theta$ and $m_\theta(G\setminus\{i,j\}) = 1$. In this case, as $\lambda_{i\sim j}^G(\theta) = -\infty$, by the expression above there exists a path $P\in [i\to j]$ such that  
\[
m_\theta(G\setminus P) < m_\theta(G\setminus\{i,j\})=1,
\]  
\noindent so $m_\theta(G\setminus P)=0$, and again $P$ is a $\theta$--critical path in $G$.  
\end{proof}

We conclude this section with two final results about paths, which will be needed in the proof of Theorem~\ref{thm:main_theorem_technical}.

For a path $P:\,i_1,\dots,i_k$ in $G$ define
\[
W_\theta(P)\defeq\bigl|\{j\in[k]\mid i_j\in\infty_\theta^{G\setminus\{i_1,\dots,i_{j-1}\}}\}\bigr|-\bigl|\{j\in[k]\mid i_j\in 0_\theta^{G\setminus\{i_1,\dots,i_{j-1}\}}\}\bigr|.
\]

\begin{lemma}[{\cite[p. 9]{spier2023refined}}]\label{lem:entry_of_path} Let $P$ be a path in $G$. Then
\[
m_\theta(G\setminus P)=m_\theta(G)+W_\theta(P).
\]
\end{lemma}
\begin{proof} This result is an immediate consequence of Equation~\eqref{eq:vertex_classification}.
\end{proof}

Note that, by Lemmas~\ref{lem:critical_path} and~\ref{lem:entry_of_path}, we always have $W_\theta(P)\geq -1$, with equality if and only if $P$ is a $\theta$--critical path of $G$. Furthermore, as explained in~\cite[Lem. 16]{spier2023refined}, the inequality $W_\theta(P)\geq -1$ can also be obtained from the fact that if $i_j\in 0_\theta^{G\setminus\{i_1,\dots, i_{j-1}\}}$, then, since $i_j$ and $i_{j-1}$ are neighbors, Lemma~\ref{lem:alpha_recursion} (see also~\cite[Lem. 15]{spier2023refined} or~\cite[Lem. 3.4]{godsil1995algebraic}) implies that $i_{j-1}\in\infty_\theta^{G\setminus\{i_1,\dots, i_{j-2}\}}$.

\begin{lemma}\label{lem:entry_of_path_start_critical} Let $P:\, i_1,\dots, i_k$ be a path in $G$, and assume that $i_1\in 0_\theta^G$ and $W_\theta(P)\geq 0$. Then there exists a minimal sub-path $\tilde{P}:\, i_1,\dots, i_\ell$ of $P$ such that $W_\theta(\tilde{P}) = 0$. Moreover, $\tilde{P}$ coincides with $P$ if and only if $k$ is the smallest index such that $i_k\in\infty_\theta^{G\setminus\{i_1,\dots, i_{k-1}\}}$.
\end{lemma}
\begin{proof} By the definition of $W_\theta(P)$, since $i_1\in 0_\theta^G$ and $W_\theta(P)\geq 0$, there exists some index $j\in [k]\setminus\{1\}$ such that $i_j\in\infty_\theta^{G\setminus\{i_1,\dots, i_{j-1}\}}$. Let $\ell\in [k]\setminus\{1\}$ be the smallest such index, and consider the sub-path $\tilde{P}:\, i_1,\dots, i_\ell$. Note that, by the discussion above, $i_j\in\pm^{G\setminus\{i_1,\dots, i_{j-1}\}}_\theta$ for any $j\in [\ell-1]\setminus\{1\}$. Therefore, we have $W_\theta(\tilde{P}) = 0$ and $W_\theta(P') = -1$, for every sub-path $P':\, i_1,\dots, i_j$ of $\tilde{P}$ with $j\in [\ell-1]$. This proves the existence of the minimal sub-path $\tilde{P}$.

Note that $\tilde{P}$ coincides with $P$ if and only if $\ell$ is equal to $k$.
\end{proof}

%%%%%%%%%%%%%%%%%%%%%%%%%%%%%%%%%%%%%%%%%%%%%%%%%

\subsection{Covers}\label{sec:covers}\

Following~\cite{gross2001topological,li2025eigenvalues,banks2022point,avni2020periodic}, we define the universal cover $G^\uni$ and the maximal abelian cover $G^\ab$ of a multi-graph $G$, together with their adjacency operators, spectra, and eigenvalues. Our construction of these covers differs slightly from that in~\cite{li2025eigenvalues} and is closer to the approach of~\cite{gross2001topological}, which is based on arc functions. In particular, unlike in~\cite{li2025eigenvalues}, we use $G^\uni$ to denote the actual universal cover of $G$.

In this section, and throughout the remainder of the paper, we fix an arbitrary orientation of the edges of $G$, and denote by $\vec{E}_+(G)$ and $\vec{E}_-(G)$ the sets of positive and negative arcs, respectively. Thus, each edge $e\in E(G)$ is represented once in $\vec{E}_+(G)$ and $\vec{E}_-(G)$, with opposite orientations. Let $\mathbb{F}_{\vec{E}_+(G)}$ be the free group generated by $\vec{E}_+(G)$. We consider the inverse of a generator as the reverse arc. In this way, the set of generators together with their inverses can be identified with the set of arcs $\vec{E}(G)$. Fix a maximal spanning forest $F$ of $G$, and let $S_+$ denote the set of positive arcs of $G$ not contained in $F$. We note that different choices of positive arcs $\vec{E}_+(G)$ and maximal spanning forest $F$ will lead to equivalent constructions. For further details we refer the reader to~\cite{gross2001topological}.

For a group $\mathcal{G}$ and a homomorphism $\phi:\Fds_{\vec{E}_+(G)}\to\mathcal{G}$, we call $\phi$ \emph{normalized} if every edge in $F$ is mapped to the identity. Now, given a group $\mathcal{G}$ and a normalized surjective homomorphism $\phi:\Fds_{\vec{E}_+(G)}\to\mathcal{G}$, we define the \emph{cover} $G^\phi$ of $G$ as follows. The vertex set is $V(G)\times\mathcal{G}$, and there is an arc joining $(o(e),g)$ to $(t(e),h)$ if $e\in\vec{E}(G)$ and $h=g\cdot\phi(e)$. Note that the reverse arc is also included with this definition. It can be shown that this construction preserves connectedness, so $G^\phi$ always has the same number of connected components as $G$~\cite[p. 91]{gross2001topological}.

If $\phi_1:\Fds_{\vec{E}_+(G)}\to\mathcal{G}_1$ and $\phi_2:\Fds_{\vec{E}_+(G)}\to\mathcal{G}_2$ are normalized surjective homomorphisms such that $\phi_2 =\pi\circ\phi_1$ for some homomorphism $\pi:\mathcal{G}_1\to\mathcal{G}_2$, then $G^{\phi_1}$ is a cover of $G^{\phi_2}$. The corresponding covering map $\Pi:G^{\phi_1}\to G^{\phi_2}$ is defined on vertices by $\Pi(i,g) = (i,\pi(g))$. Moreover, for $e\in\vec{E}(G)$, it maps the arc from $(o(e),g)$ to $(t(e),g\cdot\phi_1(e))$ in $G^{\phi_1}$ to the arc from $(o(e),\pi(g))$ to $(t(e),\pi(g)\cdot\phi_2(e))$ in $G^{\phi_2}$.

The original multi-graph $G$ corresponds, in this construction, to the trivial homomorphism $\phi_{\mathrm{triv}}:\Fds_{\vec{E}_+(G)}\to\{1\}$ that sends all generators to the identity. The \emph{universal cover} $G^\uni$ corresponds to the normalized homomorphism $\phi_\uni:\Fds_{\vec{E}_+(G)}\to\mathbb{F}_{S_+}$, where $\mathbb{F}_{S_+}$ is the free group on $S_+$, and arcs in $S_+$ are mapped to their corresponding generators in $\mathbb{F}_{S_+}$. The \emph{maximal abelian cover} $G^\ab$ corresponds to the normalized homomorphism $\phi_\ab:\Fds_{\vec{E}_+(G)}\to\mathbb{F}^\ab_{S_+}$, where $\mathbb{F}^\ab_{S_+}$ is the free abelian group on $S_+$, and arcs in $S_+$ are mapped to their corresponding generators in $\mathbb{F}^\ab_{S_+}$. In particular, as $G^{\phi}$ is always a cover of $G$ we can lift the weights of $G$ to $G^\phi$ by pull-back. In this way, $G^\uni$ and $G^\ab$ are weighted multi-graphs. 

The universal cover $G^\uni$ and maximal abelian cover $G^\ab$ can also be constructed explicitly by considering non-backtracking walks starting from a given vertex of $G$, as in~\cite{banks2022point, li2025eigenvalues}. In this way, both the universal and maximal abelian covers rooted at a given vertex are analogous to the rooted path tree introduced in Section~\ref{sec:matching_pol}. Moreover, if $G$ itself is a forest, then $F = G$ and $S_+ =\emptyset$. In this case, $G^\uni$ and $G^\ab$ are equal to $G$.

We can then define the adjacency operators $A^{G^\uni}$ and $A^{G^\ab}$, as at the start of Section~\ref{sec:covers}, which are bounded Hermitian operators on the Hilbert spaces $\ell^2(V(G^\uni))$ and $\ell^2(V(G^\ab))$, respectively. The spectra and eigenvalues of $G^\uni$ and $G^\ab$ are defined as those of the corresponding adjacency operators. Since these operators are Hermitian, the spectrum and eigenvalues of $G^\uni$ and $G^\ab$ are real. For further details, see~\cite{li2025eigenvalues, banks2022point}.

As mentioned in Section~\ref{sec:introduction}, Banks et al.~\cite{banks2022point} established, in Theorem~\ref{thm:universal_cover_eigenvalue}, a criterion for a real number $\theta$ to be an eigenvalue of $G^\uni$ in terms of $\theta$--Aomoto subsets. We now define what a \emph{(refined) $\theta$--Aomoto subset} is.

\begin{definition}[$\theta$--Aomoto subset]\label{def:theta_aomoto} A subset $S\subseteq V(G)$ is a $\theta$--Aomoto subset of $G$ if:
\begin{itemize}
    \item $G[S]$ is a forest,
    \item $\theta$ is an eigenvalue of each component of $G[S]$, and
    \item $\lvert\partial_G S\rvert <\mathrm{cc}(G[S])$.
\end{itemize}
A $\theta$--Aomoto subset is refined if, in addition, every component of $G[S]$ is $\theta$--critical and each $\emptyset\neq U\subseteq\partial_G S$ is adjacent to at least $\lvert U\rvert + 1$ components of $G[S]$.
\end{definition}

The next result shows that we can always replace a $\theta$--Aomoto subset by a refined one.

\begin{proposition}\label{prop:aomoto_implies_refined_aomoto} Every $\theta$--Aomoto subset $S$ of $G$ has a subset $\tilde{S}$ that is a refined $\theta$--Aomoto subset of $G$. Moreover, $\mathrm{cc}(G[\tilde{S}])-\lvert\partial_G \tilde{S}\rvert\geq \mathrm{cc}(G[S])-\lvert\partial_G S\rvert$.
\end{proposition}
\begin{proof} Let $S$ be a $\theta$--Aomoto subset of $G$ with components $T_1,\dots, T_k$. Thus $\lvert\partial_G S\rvert <\mathrm{cc}(G[S])=k$ and each $T_i$ is a tree with $\theta$ as an eigenvalue. By Theorem~\ref{thm:matching_equals_characteristic}, $\theta$ is a zero of $\mu^{T_i}(x)$, so $m_\theta(T_i)\geq 1$ and $0^{T_i}_\theta\neq\emptyset$ for each $i$. 

Define $S'\defeq\bigcup_{i\in [k]}0_\theta^{T_i}$. Clearly $\emptyset\neq S'\subseteq S$. Note that $G[S']$ is a forest whose components are the union of the $\theta$--critical components of $T_i$ over $i$ in $[k]$. Therefore, by Theorem~\ref{thm:stability}, the components of $G[S']$ are $\theta$--critical. We claim that $\mathrm{cc}(G[S'])-\lvert\partial_G S'\rvert\geq\mathrm{cc}(G[S]) -\lvert\partial_G S\rvert > 0$, and hence $S'$ is a $\theta$--Aomoto subset of $G$.

For each $i$ in $[k]$, Theorem~\ref{thm:multiplicity} implies 
\[
\mathrm{cc}(0_\theta^{T_i})-\lvert\partial_{T_i}0_\theta^{T_i}\rvert=m_\theta(T_i)\geq 1.
\]
Also note that $\mathrm{cc}(G[S'])=\sum_{i\in [k]}\mathrm{cc}(0_\theta^{T_i})$ and \(\partial_G S'\subseteq\partial_G S\cup\bigcup_{i\in [k]}\partial_{T_i}0_\theta^{T_i}\).
Therefore,
\begin{align*}
\mathrm{cc}(G[S'])- \lvert\partial_G S'\rvert&\geq k+\sum_{i\in [k]}\lvert\partial_{T_i}0_\theta^{T_i}\rvert- \lvert\partial_G S'\rvert\\&=\mathrm{cc}(G[S])+\sum_{i\in [k]}\lvert\partial_{T_i}0_\theta^{T_i}\rvert- \lvert\partial_G S'\rvert\\&=\left(\mathrm{cc}(G[S])-\lvert \partial_G S\rvert\right)+\lvert \partial_G S\rvert+\sum_{i\in [k]}\lvert\partial_{T_i}0_\theta^{T_i}\rvert- \lvert\partial_G S'\rvert\\&\geq\left(\mathrm{cc}(G[S])-\lvert \partial_G S\rvert\right)+\Bigl\lvert\partial_G S\cup\bigcup_{i\in [k]}\partial_{T_i}0_\theta^{T_i}\Bigr\rvert- \lvert\partial_G S'\rvert\\
&\geq \left(\mathrm{cc}(G[S])-\lvert \partial_G S\rvert\right)\\&>0.
\end{align*}
This proves our claim.

Now observe that if $U, V\subseteq\partial S'$ are adjacent to at most $\lvert U\rvert$ and $\lvert V\rvert$ components of $G[S']$, respectively, then $U\cup V$ is adjacent to at most $\lvert U\cup V\rvert$ components of $G[S']$. Let $U'$ denote the union of all subsets $U\subseteq\partial_G S'$ that are adjacent to at most $\lvert U\rvert$ components of $G[S']$. Then $U'$ itself is adjacent to at most $\lvert U'\rvert$ components of $G[S']$.

Consider $\tilde{S}\subseteq S'$ defined as the union of the components of $G[S']$ not adjacent to $U'$. We claim that $\tilde{S}$ is a refined $\theta$--Aomoto subset of $G$ and $\mathrm{cc}(G[\tilde{S}])-\lvert\partial_G \tilde{S}\rvert\geq \mathrm{cc}(G[S'])-\lvert\partial_G S'\rvert>0$. Note that since $\partial_G \tilde{S} \subseteq \partial_G S' \setminus U'$, it follows that $\lvert \partial_G \tilde{S} \rvert \leq \lvert \partial_G S' \rvert - \lvert U' \rvert$. Furthermore, because $U'$ is adjacent to at most $\lvert U' \rvert$ components of $G[S']$, we obtain $\mathrm{cc}(G[\tilde{S}]) \geq \mathrm{cc}(G[S']) - \lvert U' \rvert$. Therefore,
\begin{align*}
\mathrm{cc}(G[\tilde{S}]) - \lvert \partial_G \tilde{S} \rvert &\geq (\mathrm{cc}(G[S']) - \lvert U' \rvert) - (\lvert \partial_G S' \rvert - \lvert U' \rvert) \\&\geq \mathrm{cc}(G[S']) - \lvert \partial_G S' \rvert \\& > 0.
\end{align*}

Moreover, $G[\tilde{S}]$ is a forest, and every component of $G[\tilde{S}]$ is $\theta$--critical. Now, given $\emptyset\neq U\subseteq\partial_G\tilde{S}$, if $U$ were adjacent to at most $\lvert U\rvert$ components of $G[\tilde{S}]$, then $U\cup U'\supsetneq U'$ would be adjacent to at most $\lvert U\cup U'\rvert$ components of $G[S']$, contradicting our choice of $U'$. Hence, $U$ is adjacent to at least $\lvert U\rvert +1$ components of $G[\tilde{S}]$. This proves our claim and completes the proof.
\end{proof}

As observed in~\cite[p. 33]{spier2021graph}, Theorem~\ref{thm:universal_cover_eigenvalue} can be viewed as a version of the $\theta$--Gallai-Edmonds decomposition where the path tree is replaced by the universal cover. Indeed, the definition of (refined) $\theta$--Aomoto subsets mirrors the definition of $\theta$--critical components in Section~\ref{sec:matching_pol} and has a property analogous to the one stated in Theorem~\ref{thm:matched_special}. We now describe another aspect of this analogy, this time with respect to the \emph{density of states} of $G^\uni$.

Let $\Pi_\uni: G^\uni \to G$ be the covering map from $G^\uni$ to $G$.  For each $j \in V(G^\uni)$, let $\chi_j$ denote the characteristic vector of $j$ on $\ell^2(V(G^\uni))$. For every vertex $i \in V(G)$ and $n \in \Nds$, $\langle \chi_{\tilde{i}}, (A^{G^\uni})^n \chi_{\tilde{i}} \rangle$ is real and constant for all $\tilde{i} \in \Pi_\uni^{-1}(i)$.  

Given a vertex $i \in V(G)$, the \emph{spectral measure} $\tau_i$ is the unique measure on $\Rds$ satisfying  
\[
\langle \chi_{\tilde{i}}, (A^{G^\uni})^n \chi_{\tilde{i}} \rangle = \int_{\Rds} x^n \, d\tau_i(x),
\]
for every $\tilde{i} \in \Pi_\uni^{-1}(i)$ and $n \in \Nds$.

The \emph{density of states} of $G^\uni$ is then defined as the measure obtained by averaging the measures $\tau_i$ over all vertices $i \in V(G)$:  
\[
\tau \defeq \frac{1}{\lvert V(G)\rvert} \sum_{i \in V(G)} \tau_i.
\]  

It was shown in~\cite{banks2022point} that for every eigenvalue $\theta$ of $G^\uni$, one can associate a specific $\theta$--Aomoto subset of $G$ that is related to $\tau(\theta)$. For a real number $\theta$, let $X_\theta^G$ denote the set of vertices $i \in V(G)$ such that $\tau_i(\theta) \neq 0$.  The following result was established by Banks et al.~\cite{banks2022point}.

\begin{theorem}[{\cite[Thm. 3.1]{banks2022point}}]\label{thm:the_aomoto_set} Let $\theta$ be an eigenvalue of $G^\uni$. Then $X_\theta^G$ is a $\theta$--Aomoto subset of $G$, and $\tau(\theta) = \mathrm{cc}(G[X_\theta^G])-\lvert \partial_G X_\theta^G \rvert$.
\end{theorem}

Moreover, Banks et al.~\cite{banks2022point} showed that $\tau(\theta)$ can be computed as the solution of a maximization problem over $\theta$--Aomoto subsets. For the next result, let $\mathcal{A}_\theta^G$ denote the set of $\theta$--Aomoto subsets of $G$.

\begin{theorem}[{\cite[Cor. 3.4 and Thm. 3.1]{banks2022point}}]\label{thm:density_of_states} Let $\theta$ be an eigenvalue of $G^\uni$. Then  
\[
\tau(\theta) = \max_{S \in \mathcal{A}_\theta^G} \frac{\mathrm{cc}(G[S]) - \lvert \partial_G S \rvert}{\lvert V(G) \rvert}.
\]  
Moreover, $X_\theta^G$ is a maximizer.
\end{theorem}

As observed in~\cite[Fig.~2]{banks2022point}, although $X_\theta^G$ is always a maximizer in Theorem~\ref{thm:density_of_states}, it may not be the unique one. We address this issue in Theorem~\ref{thm:density_of_states_refined}, which will be proved in Section~\ref{sec:proof_of_main_results}.

Theorems~\ref{thm:the_aomoto_set} and ~\ref{thm:density_of_states} should be compared with Theorem~\ref{thm:multiplicity} and~\cite[Thm. 2.1]{ku2010extensions}, respectively. A final, explicit connection between refined $\theta$--Aomoto subsets and $\theta$--critical components will be presented in Corollary~\ref{cor:refined_aomoto_in_gallai_edmonds}.

As stated in Section~\ref{sec:introduction}, Li et al.~\cite{li2025eigenvalues} recently obtained, in Theorem~\ref{thm:abelian_cover_eigenvalue}, a characterization of the eigenvalues of $G^\ab$ in terms of the base multi-graph $G$. To describe their results, for $\xi\in\mathbb{T}^{\vec{E}_+(G)}$, let $A^G_\xi$ be the Hermitian matrix given by  
\[
(A^G_\xi)_{ij}\defeq r_i \delta_{ij} + \sum_{\substack{e \in \vec{E}(G)\\ o(e)=i, t(e)=j}} \rho_e\xi_e, \qquad \text{for all } i,j \in V(G),
\]
\noindent where $\xi_e\defeq\overline{\xi_{e^{-1}}}$ for $e\in \vec{E}_-(G)$.

We define $\phi^G_\xi(x)$ as the characteristic polynomial of $A^G_\xi$.

\begin{theorem}[{\cite[Prop. 2.4]{li2025eigenvalues}}]\label{thm:abelian_cover_delta} A real number $\theta$ is an eigenvalue of $G^\ab$ if and only if $\theta$ is a zero of $\phi^G_\xi(x)$ for every $\xi\in\Tds^{\vec{E}_+(G)}$.
\end{theorem}

With this characterization of the eigenvalues of $G^\ab$, Li et al.~\cite{li2025eigenvalues} further proved the following result.

\begin{theorem}[{\cite[Sec. 3.2]{li2025eigenvalues}}]\label{thm:abelian_cover_equivalences} For a real number $\theta$, the following are equivalent:  
\begin{enumerate}[(a)]
    \item $\theta$ is a zero of $\phi^G_\xi(x)$ for every $\xi\in\mathbb{T}^{\vec{E}(G)}$;  
    \item $\theta$ is a zero of $\mu^{G\setminus\Gamma}(x)$ for every 2-regular subgraph $\Gamma$ of $G$;  
    \item $\theta$ is a zero of $\phi^{G\setminus\Gamma}(x)$ for every 2-regular subgraph $\Gamma$ of $G$.  
\end{enumerate}  
\end{theorem}  

As a consequence of Theorems~\ref{thm:abelian_cover_delta} and~\ref{thm:abelian_cover_equivalences}, Li et al.~\cite{li2025eigenvalues} obtained Theorem~\ref{thm:abelian_cover_eigenvalue}. In the course of establishing these results, they also proved the following theorem, which, together with Theorems~\ref{thm:universal_cover_eigenvalue} and~\ref{thm:abelian_cover_eigenvalue}, implies Theorem~\ref{thm:universal_cover_implies_abelian_cover}.

\begin{theorem}[{\cite[Prop. A.2]{li2025eigenvalues}}]\label{thm:aomoto_implies_matching} Assume that $G$ has a $\theta$--Aomoto subset. Then, $\theta$ is a zero of $\mu^{G\setminus\Gamma}(x)$ for every $2$-regular subgraph $\Gamma$.
\end{theorem}

Part of this result was claimed without proof by the author in~\cite[p. 33]{spier2021graph}. Indeed, it was stated there that if $G$ is a graph, then every eigenvalue of $G^\uni$ is a zero of $\phi^G(x)$, $\mu^G(x)$, and, more generally, of every molecular polynomial associated with $G$. In the next section, we make this statement precise in Corollary~\ref{cor:aomoto_implies_molecular}.

While not stated in~\cite{spier2021graph}, the author had in fact established Theorem~\ref{thm:aomoto_implies_matching} for graphs in the course of proving Corollary~\ref{cor:aomoto_implies_molecular}. This alternative proof differs from that in~\cite{li2025eigenvalues} and follows from the results stated in Section~\ref{sec:new_gallai_edmonds}.

%%%%%%%%%%%%%%%%%%%%%%%%%%%%%%%%%%%%%%%%%%%%%%%%%

\subsection{Molecular polynomial}\

We introduce \emph{molecular polynomials}, which generalize both the characteristic and matching polynomials. These were first defined as $\mu$-polynomials by Gutman and Polansky~\cite{gutman1981cyclic, gutman1986some}. Since the notation $\mu$-polynomial may lead to confusion with the matching polynomial, we adopt the term \emph{molecular polynomial}.

In this section, for a multi-graph $G$ as defined at the beginning of Section~\ref{sec:preliminaries}, we assign weights $\lambda_C \in \Cds$ to each of its directed cycles $C$. By a directed cycle of $G$, we mean a sequence of arcs $e_1, \dots, e_k$ such that $t(e_j) = o(e_{j+1})$ for every $j \in [k-1]$ and $t(e_k) = o(e_1)$. Directed cycles of $G$ therefore correspond to undirected cycles of $G$ with a chosen orientation, as well as to pairs of distinct arcs in opposite directions, which form directed cycles of length $2$. We refer to this latter type of directed cycle as being associated with edges of $G$.

For each directed cycle associated with an edge $e \in E(G)$ that is not a loop, we define $\lambda_e\defeq -\lvert\rho_e\rvert^2$, in accordance with the definition given in Section~\ref{sec:matching_pol}. For all remaining directed cycles of $G$, the weights are chosen arbitrarily in $\Cds$, subject only to the condition $\lambda_{C^{-1}} = \overline{\lambda_C}$, where $C^{-1}$ denotes the directed cycle $C$ with the reverse direction.

We denote by $\mathcal{P}_G$ the collection of subgraphs of $G$ that are disjoint unions of directed cycles of $G$. For $\Gamma \in \mathcal{P}_G$, we write $C \subseteq \Gamma$ when the directed cycle $C$ is a subgraph of $\Gamma$. The \emph{molecular polynomial} of $G$ is defined by  
\[
M^G(x)\defeq\sum_{\Gamma\in\mathcal{P}_G}\prod_{i \notin V(\Gamma)}(x-r_i)\prod_{C\subseteq\Gamma}\lambda_C.
\]

This polynomial generalizes both the matching and characteristic polynomials. Indeed, if we set $\lambda_C=0$ for every directed cycle which is not associated with an edge of $G$, then $M^G(x)$ coincides with $\mu^G(x)$. On the other hand, if we set $\lambda_C$ to be the product $-\prod_{e\in C}\rho_e$ for every directed cycle $C$, then by Harary's formula~\cite{harary1962determinant}, $M^G(x)$ coincides with $\phi^G(x)$. Finally, given $\xi\in\mathbb{T}^{\vec{E}_+(G)}$, as in Section~\ref{sec:covers}, we extend it to $\tilde{\xi}\in\mathbb{T}^{\vec{E}(G)}$ by setting $\tilde{\xi}_e =\xi_e$ if $e\in\vec{E}_+(G)$ and $\tilde{\xi}_e =\overline{\xi_{e^{-1}}}$ if $e\in\vec{E}_-(G)$. For a directed cycle $C$, we then define $\lambda_C = -\prod_{e\in\vec{E}(C)}\rho_e\tilde{\xi}_e$. With this choice of weights, $M^G(x)$ is then equal to $\phi^G_\xi(x)$.

Given a multi-graph $G$ as defined at the start of Section~\ref{sec:preliminaries}, a \emph{molecular polynomial associated with $G$} is defined as any polynomial that can be obtained as $M^G(x)$ for some choice of weights for the directed cycles of $G$ that are not associated with edges of $G$. Note that if $G$ is a forest, then the only molecular polynomial associated with $G$ is $\mu^G =\phi^G$, in accordance with Theorem~\ref{thm:matching_equals_characteristic}.

The following lemma is an immediate consequence of the observations above.

\begin{lemma}\label{lem:molecular_implies_delta} 
If $\theta$ is a zero of every molecular polynomial $M^G(x)$ associated with $G$, then $\theta$ is a zero of $\phi^G_\xi(x)$ for every $\xi\in\Tds^{\vec{E}(G)}$.
\end{lemma}

The following result was obtained by Gutman~\cite{gutman1986some}. For the next statement, let $\tilde{\mathcal{P}}_G \subseteq \mathcal{P}_G$ denote the set of subgraphs of $G$ that are disjoint unions of directed cycles of $G$, none of which is associated with an edge of $G$.

\begin{theorem}[{\cite[Prop. 1a]{gutman1986some}}]\label{thm:molecular_to_matching} For every multi-graph $G$,  
\[
M^G(x) =\sum_{\tilde{\Gamma}\in\tilde{\mathcal{P}}_G}\mu^{G\setminus\tilde{\Gamma}}(x)\prod_{C\in\tilde{\Gamma}}\lambda_C.
\] 
\end{theorem}

As a corollary, we obtain the following result.

\begin{corollary}\label{cor:matching_implies_molecular} If $\theta$ is a zero of $\mu^{G\setminus\Gamma}(x)$ for every $2$-regular subgraph $\Gamma$ of $G$, then $\theta$ is a zero of every molecular polynomial $M^G(x)$ associated with $G$.
\end{corollary}

Corollary~\ref{cor:matching_implies_molecular} together with Theorem~\ref{thm:aomoto_implies_matching} implies the following result mentioned in Section~\ref{sec:covers}.

\begin{corollary}\label{cor:aomoto_implies_molecular} Assume that $G$ has a $\theta$--Aomoto subset. Then $\theta$ is a zero of every molecular polynomial $M^G(x)$ associated with $G$.
\end{corollary}

By combining Lemma~\ref{lem:molecular_implies_delta} with Corollary~\ref{cor:matching_implies_molecular} and Theorems~\ref{thm:abelian_cover_delta} and~\ref{thm:abelian_cover_equivalences}, we also obtain the following statement.

\begin{proposition}\label{prop:abelian_cover_molecular} A real number $\theta$ is an eigenvalue of $G^\ab$ if and only if it is a zero of every molecular polynomial $M^G(x)$ associated with $G$. 
\end{proposition}

%%%%%%%%%%%%%%%%%%%%%%%%%%%%%%%%%%%%%%%%%%%%%%%%%%%%%%%%%%%%%%%%%%%%%%%%%%%%%%%%

\section{Main results}\label{sec:main_results}\

We now state Theorem~\ref{thm:main_theorem_introduction_1} in its final form, providing several equivalent characterizations of the eigenvalues of $G^\uni$ and $G^\ab$.

\begin{theorem}\label{thm:main_theorem} Let $\theta$ be a real number. Then the following are equivalent:
\begin{enumerate}[(a)]
    \item $\theta$ is an eigenvalue of $G^\uni$;
    \item $\theta$ is an eigenvalue of $G^\ab$;
    \item $G$ has a $\theta$--Aomoto subset;
    \item $G$ has a refined $\theta$--Aomoto subset;
    \item $\theta$ is a zero of $\mu^{G\setminus\Gamma}(x)$ for every $2$-regular subgraph $\Gamma$ of $G$;
    \item $\theta$ is a zero of $\phi^{G\setminus\Gamma}(x)$ for every $2$-regular subgraph $\Gamma$ of $G$;
    \item $\theta$ is a zero of $\phi_\xi^{G}(x)$ for every $\xi\in\Tds^{\vec{E}_+(G)}$;
    \item $\theta$ is a zero of every molecular polynomial $M^G(x)$ associated with $G$.
\end{enumerate}
\end{theorem}

Note that Theorem~\ref{thm:universal_cover_eigenvalue} and Proposition~\ref{prop:aomoto_implies_refined_aomoto} imply the equivalence of $(a)$, $(c)$, and $(d)$. Moreover, Theorems~\ref{thm:abelian_cover_delta} and~\ref{thm:abelian_cover_equivalences}, together with Proposition~\ref{prop:abelian_cover_molecular}, establish the equivalences of $(b)$, $(e)$, $(f)$, $(g)$, and $(h)$. In addition, Theorem~\ref{thm:aomoto_implies_matching} shows that $(c)$ implies $(e)$. Finally, Theorem~\ref{thm:main_theorem_introduction_2} will prove that $(f)$ implies $(c)$, completing the proof.

The version of Theorem~\ref{thm:main_theorem_introduction_2} that we are going to prove is as follows.

\begin{theorem}\label{thm:main_theorem_technical} Assume that $G$ has no $\theta$--Aomoto subset and let $k\defeq m_\theta(G)$. Then there are $k$ disjoint cycles $C_1,\dots, C_k$ in $G$ such that 
\[
m_\theta(G\setminus (C_1\sqcup\cdots\sqcup C_k)) = 0.
\]
\end{theorem}

Note that by Corollary~\ref{cor:critical_cycle} at least $k$ cycles are required in this statement. 

Theorem~\ref{thm:main_theorem_technical} is proved by induction on $k=m_\theta(G)$ in two steps. The first consists of finding a $\theta$--critical cycle $C$, assuming $G$ has no $\theta$--Aomoto subset. The second, more involved step, is to show that deleting this cycle does not create a $\theta$--Aomoto subset in $G\setminus C$.

As mentioned in Section~\ref{sec:covers}, although $X_\theta^G$, the $\theta$–Aomoto subset defined in terms of the spectral measures, is always a maximizer in Theorem~\ref{thm:density_of_states}, it may not be the unique one. Our next result shows that if the maximization is performed over refined $\theta$--Aomoto subsets, then there is a unique maximizer. For this result, let $\mathcal{R}_\theta^G$ denote the set of refined $\theta$--Aomoto subsets of $G$.

\begin{theorem}\label{thm:density_of_states_refined} Let $\theta$ be an eigenvalue of $G^\uni$. Then,
\[
\tau(\theta) =\max_{S\in \mathcal{R}_\theta^G}\frac{\mathrm{cc}(G[S]) -\lvert\partial_G S\rvert}{\lvert V(G)\rvert},
\]  
Moreover, there is a unique maximizer, which is the maximal refined $\theta$--Aomoto subset of $G$.
\end{theorem}

It is not clear whether $X_\theta^G$ is a maximizer in Theorem~\ref{thm:density_of_states_refined}, since it has not been shown that it is a refined $\theta$--Aomoto subset. By~\cite[Thm. 3.1 and Lem. 4.2]{banks2022point}, the only property missing for $X_\theta^G$ to be a refined $\theta$--Aomoto subset is that if $\emptyset \neq U \subseteq \partial_G X_\theta^G$, then $U$ is adjacent to at least $\lvert U \rvert + 1$ components of $G[X_\theta^G]$. We believe that, for instance by adapting the proof of~\cite[Thm. 1.11]{bencs2022atoms} to weighted infinite multi-graphs, this property could be established. If this is indeed the case, $X_\theta^G$ would be characterized as the maximal refined $\theta$--Aomoto subset of $G$ and as the unique maximizer in Theorem~\ref{thm:density_of_states_refined}.

%%%%%%%%%%%%%%%%%%%%%%%%%%%%%%%%%%%%%%%%%%%%%%%%%%%%%%%%%%%%%%%%%%%%%%%%%%%%%%%%

\section{New results on $\theta$--Gallai-Edmonds decomposition}\label{sec:new_gallai_edmonds}\

We begin by studying (refined) $\theta$--Aomoto subsets and by providing an alternative proof of Theorem~\ref{thm:aomoto_implies_matching}. These results will also be required for the proof of Theorems~\ref{thm:main_theorem_technical} and~\ref{thm:density_of_states_refined}.

\begin{proposition}\label{prop:robust_refined_aomoto} Assume that $G$ has a $\theta$--Aomoto subset and let $C$ be a cycle in $G$. Then, $G\setminus C$ has a $\theta$--Aomoto subset. 
\end{proposition}
\begin{proof} Let $S$ be a $\theta$--Aomoto subset of $G$ with components $T_1,\dots, T_k$. Let $I\subseteq [k]$ be the set of indices $i$ in $[k]$ such that $C$ intersects $T_i$, and define $\tilde{S}\defeq\bigsqcup_{j\in [k]\setminus I} V(T_j)$. We claim that $\tilde{S}$ is a $\theta$--Aomoto subset of $G\setminus C$. It suffices to prove that $\lvert\partial_{G\setminus C}\tilde{S}\rvert <\mathrm{cc}\big((G\setminus C)[\tilde{S}]\big) = k -\lvert I\rvert$, that is,\[
\lvert I\rvert +\lvert\partial_{G\setminus C}\tilde{S}\rvert < k.
\]

Observe that $|I|\leq |\partial_G S\cap C|$. Indeed, since each $T_j$ is a tree, $C$ cannot be contained entirely within any component $T_j$. Moreover, in order to travel between any two components $T_i$ and $T_j$, the cycle $C$ must pass through $\partial_G S$. Note also that $\partial_{G\setminus C}\tilde{S}=\partial_G\tilde{S}\setminus C\subseteq\partial_{G}S\setminus C$, since $\tilde{S}\cap C=\emptyset$ and $\tilde{S}$ is the union of some components of $G[S]$.  

Hence,  
\begin{align*}
\lvert I\rvert +\lvert\partial_{G\setminus C}\tilde{S}\rvert &\leq\lvert\partial_G S\cap C\rvert +\lvert\partial_{G\setminus C}\tilde{S}\rvert\\&\leq\lvert\partial_G S\cap C\rvert +\lvert\partial_{G}S\setminus C\rvert\\ &=\lvert\partial_G S\rvert\\ &< k,  
\end{align*}  
\noindent proving our claim.
\end{proof}

The next proposition is technical, but fundamental to the results that follow.

\begin{proposition}\label{prop:subset_in_gallai_edmonds} Let $S\subseteq V(G)$ such that:
\begin{itemize}
    \item each component of $G[S]$ is $\theta$--critical, and
    \item every $\emptyset\neq U\subseteq\partial_G S$ is adjacent to at least $\lvert U\rvert + 1$ components of $G[S]$.
\end{itemize}
Then the components of $G[S]$ are $\theta$--critical components of $G$. Moreover, if $Z$ is a subset of vertices of a component of $G[S]$, then $\partial_G S\subseteq\infty^{G\setminus Z}_\theta$.
\end{proposition}
\begin{proof} We proceed by induction on $k\defeq\lvert\partial_G S\rvert$.  

If $k = 0$, then the components of $G[S]$ coincide with components of $G$. Since each component of $G[S]$ is $\theta$--critical, it follows that they are $\theta$--critical components of $G$. 

Now assume $k\geq 1$ and, by induction hypothesis, that the statement holds for every multi-graph $\tilde{G}$ and subset $\tilde{S}\subseteq V(\tilde{G})$ with $\lvert\partial_{\tilde{G}}\tilde{S}\rvert < k$. Since $\partial_G 0^G_\theta\subseteq \infty^G_\theta$, it suffices to show that $S\subseteq 0^G_\theta$ and $\partial_G S\subseteq\infty^{G\setminus Z}_\theta$ for every subset $Z$ of some component $H$ of $G[S]$. 

Consider a subset $Z$ of some component $H$ of $G[S]$, and let $j\in\partial_G S$ be arbitrary. Suppose $U, V\subseteq\partial_{G\setminus j} S$ are both adjacent to $H$ and each is adjacent to exactly $\lvert U\rvert + 1$ and $\lvert V\rvert + 1$ components of $G[S]$, respectively. Then $U\cup V$ is adjacent to $H$ and to exactly $\lvert U\cup V\rvert + 1$ components of $G[S]$. Define $U'$ as the union of all subsets $U\subseteq\partial_{G\setminus j} S$ that are adjacent to $H$ and to exactly $\lvert U\rvert + 1$ components of $G[S]$. By construction, if $U'\neq\emptyset$, then $U'$ is adjacent to $H$ and to exactly $\lvert U'\rvert + 1$ components of $G[S]$.

Consider $\tilde{S}\subseteq S$, defined as the union of the components of $G[S]$ that are not adjacent to $U'$, and let $\tilde{G}\defeq G\setminus (\{j\}\cup Z)$. Since $\mathrm{cc}(G[S])\geq\lvert\partial_G S\rvert + 1 =\lvert\partial_{G\setminus j} S\rvert + 2$, and because $U'$ (when nonempty) is adjacent to exactly $\lvert U'\rvert + 1$ components of $G[S]$, it follows that at least one component of $G[S]$ is not adjacent to $U'$. Therefore, $\tilde{S}\neq\emptyset$. 

We claim that $\tilde{S}\subseteq V(\tilde{G})$ satisfies the hypothesis of the statement and that $\lvert\partial_{\tilde{G}}\tilde{S}\rvert < k$, so $\tilde{S}\subseteq 0^{\tilde{G}}_\theta$. Indeed, each component of $\tilde{G}[\tilde{S}]$ is $\theta$--critical. Moreover, if $\emptyset\neq U\subseteq\partial_{\tilde{G}}\tilde{S}$, then $U\cup U'\supsetneq U'$, so $U\cup U'$ is adjacent to at least $\lvert U\cup U'\rvert + 2$ components of $G[S]$. Since $U'$ (when nonempty) is adjacent to exactly $\lvert U'\rvert + 1$ such components, it follows that $U$ is adjacent to at least $\lvert U\rvert + 1$ components of $G[S]$ that are not adjacent to $U'$, that is, components of $\tilde{G}[\tilde{S}]$. Finally, as $\partial_{\tilde{G}}\tilde{S}\subseteq\partial_{G\setminus j} S$, we conclude that $\lvert\partial_{\tilde{G}}\tilde{S}\rvert < k$. This proves the claim.

Now, observe that $U'\cup\{j\}$ is adjacent to at least $\lvert U'\cup\{j\}\rvert + 1 =\lvert U'\rvert + 2$ components of $G[S]$. Since $U'$ (when nonempty) is adjacent to exactly $\lvert U'\rvert + 1$ components of $G[S]$, it follows that $j$ is adjacent to at least one component of $G[S]$ not adjacent to $U'$, that is, a component of $\tilde{G}[\tilde{S}]$. As $\tilde{S}\subseteq 0^{\tilde{G}}_\theta$ and $j$ is adjacent to a vertex of $\tilde{S}$ in $\tilde{G}=G\setminus (Z\cup\{j\})$, Lemma~\ref{lem:alpha_recursion} (see also~\cite[Lem. 15]{spier2023refined} or~\cite[Lem. 3.4]{godsil1995algebraic}) implies that $j\in\infty^{G\setminus Z}_\theta$. Since $j\in\partial_G S$ is arbitrary, it follows that $\partial_G S\subseteq\infty^{G\setminus Z}_\theta$ for every subset $Z$ of a component $H$ of $G[S]$.

Now, fix a vertex $j \in \partial_G S$, and let $i \in S$ be arbitrary. Observe that $S$ and $G\setminus j$ satisfy the hypotheses of the statement, and since $\lvert\partial_{G\setminus j} S\rvert < k$, we have $S\subseteq 0^{G\setminus j}_\theta$ by the induction hypothesis. Thus we have $i\in S\subseteq 0^{G\setminus j}_\theta$. Moreover, by applying the previously established result with $Z =\{i\}$, we obtain $j\in\infty^{G\setminus Z}_\theta =\infty^{G\setminus i}_\theta$. Therefore, $i\in 0^{G\setminus j}_\theta$ and $j\in\infty^{G\setminus i}_\theta$. By~\cite[Prop. 20--24 or Fig. 5]{spier2023refined}, it follows that $i\in 0^G_\theta$. Since $i\in S$ was arbitrary, we conclude that $S\subseteq 0^G_\theta$.
\end{proof}

By Theorem~\ref{thm:matched_special}, we may apply Proposition~\ref{prop:subset_in_gallai_edmonds} with $S = 0^G_\theta$, which yields $\partial_G 0^G_\theta\subseteq\infty^{G\setminus Z}_\theta$ whenever $Z$ is a subset of a $\theta$--critical component of $G$. Moreover, Proposition~\ref{prop:subset_in_gallai_edmonds} immediately implies the following result.

\begin{corollary}\label{cor:refined_aomoto_in_gallai_edmonds} The following statements are equivalent: 
\begin{enumerate}[(a)]
    \item $S$ is a refined $\theta$--Aomoto subset of $G$; 
    \item $S = T_1\sqcup\cdots\sqcup T_k$, where $T_1,\dots, T_k$ are $\theta$--critical components of $G$ that are trees, and every $\emptyset\neq U\subseteq\partial_G S$ is adjacent to at least $\lvert U\rvert + 1$ of the trees $T_i$. 
\end{enumerate}
\end{corollary}

This result shows the connection mentioned in Section~\ref{sec:covers} between the $\theta$--Gallai-Edmonds decomposition and refined $\theta$--Aomoto subsets. Corollary~\ref{cor:refined_aomoto_in_gallai_edmonds} provides the main motivation for our approach in this work. Moreover, together with Propositions~\ref{prop:aomoto_implies_refined_aomoto} and~\ref{prop:robust_refined_aomoto}, Corollary~\ref{cor:refined_aomoto_in_gallai_edmonds} yields an alternative proof of Theorem~\ref{thm:aomoto_implies_matching}, as mentioned in Section~\ref{sec:covers}.

We also note that the analogue of Corollary~\ref{cor:refined_aomoto_in_gallai_edmonds} for the characteristic polynomial does not hold, even in the case of graphs, as illustrated by the following example.

\begin{example}\label{exm:charac_fails} Consider the graph $G$ shown in Figure~\ref{fig:graph_charac_fails}, where the weights are specified by the following matrix:
\[
A^G =
\begin{pmatrix}
0 & 1 & 1 & 0 & 0 \\
1 & 0 & 1 & 0 & 0 \\
1 & 1 & 1 & 1 & 1 \\
0 & 0 & 1 & 0 & 1 \\
0 & 0 & 1 & 1 & 0
\end{pmatrix}
\]

Note that $\phi^G(x)=(x+1)^3(x-1)(x-3)$, and  
\[
\phi^{G\setminus 1}(x)=\phi^{G\setminus 2}(x)=\phi^{G\setminus 4}(x)=\phi^{G\setminus 5}(x)= (x+1)^2(x^2-3x+1),
\]  
\noindent while $\phi^{G\setminus 3}(x)=(x+1)^2(x-1)^2$.  

Also, note that $S\defeq\{1,2,4,5\}$ is a refined $(-1)$--Aomoto subset of $G$, with $\partial_G S=\{3\}$. Clearly, for each $i\in S$, the multiplicity of $-1$ as a zero of $\phi^{G\setminus i}(x)$ is strictly smaller than its multiplicity in $\phi^G(x)$. However, in contrast to Corollary~\ref{cor:refined_aomoto_in_gallai_edmonds}, we also observe that the same holds for $i=3\in\partial_G S$.
\end{example}

\begin{figure}[h]
\centering
\begin{tikzpicture}[
    every node/.style={circle, fill=black, inner sep=1.5pt},
    every edge/.style={thick}
]
\node (3) at (0,0) [label=below:{3}] {};  
\node (1) at (-2,1.2) [label=above left:{1}] {};
\node (2) at (-2,-1.2) [label=below left:{2}] {};
\node (4) at (2,1.2) [label=above right:{4}] {};
\node (5) at (2,-1.2) [label=below right:{5}] {};
\draw (1) -- (2) -- (3) -- (1);
\draw (3) -- (4) -- (5) -- (3);
\end{tikzpicture}
\caption{Graph corresponding to Example~\ref{exm:charac_fails}.}
\label{fig:graph_charac_fails}
\end{figure}
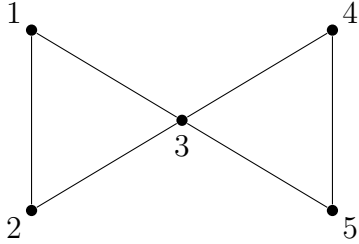

The next result shows that deleting any subset of a $\theta$--critical component of a multi-graph $G$ affects the multiplicity of $\theta$ in the same way for both the $\theta$--critical component and $G$.

\begin{proposition}\label{prop:critical_remove} Let $Z\subseteq V(H)$ for a $\theta$--critical component $H$ of $G$. Then $m_\theta(G\setminus Z)-m_\theta(G)=m_\theta(H\setminus Z)-m_\theta(H)$.
\end{proposition}
\begin{proof} We proceed by induction on $k\defeq\lvert\partial_G 0^G_\theta\rvert$.

If $k = 0$, then the $\theta$–critical components of $G$ are isolated, and
the result follows immediately. Now, assume $k\geq 1$ and, by induction hypothesis, that the statement holds for any multi-graph $\tilde{G}$ with $\lvert\partial_{\tilde{G}} 0^{\tilde{G}}_\theta\rvert < k$. 

Consider $j\in\partial_G 0^G_\theta$. By Theorem~\ref{thm:stability} and the induction hypothesis, we have $m_\theta((G\setminus j)\setminus Z)-m_\theta(G\setminus j)=m_\theta(H\setminus Z)-m_\theta(H)$. On the other hand, by Proposition~\ref{prop:subset_in_gallai_edmonds}, since $Z$ is contained in a $\theta$--critical component of $G$ and $0^G_\theta$ satisfies the hypotheses of the proposition, we have $j\in\infty^{G\setminus Z}_\theta$, which implies $m_\theta((G\setminus Z)\setminus j)=m_\theta(G\setminus Z)+1$. Also note that $m_\theta(G\setminus j)=m_\theta(G)+1$, since $j\in\partial_G 0^G_\theta\subseteq\infty^G_\theta$.

Combining these equalities, we obtain 
\begin{align*}
m_\theta(G \setminus Z)-m_\theta(G)&=m_\theta((G\setminus Z) \setminus j)-m_\theta(G \setminus j)\\&=m_\theta((G\setminus j) \setminus Z)-m_\theta(G \setminus j)\\&=m_\theta(H\setminus Z)-m_\theta(H).  
\end{align*}

This completes the induction step and the proof.
\end{proof}

The next corollary shows that $\theta$--critical cycles in a multi-graph correspond to $\theta$--critical cycles within its $\theta$--critical components.

\begin{corollary}\label{cor:critical_in_comp_equiv_critical} Let $C$ be a cycle in $G$. Then $C$ is a $\theta$--critical cycle of $G$ if and only if $C$ is a $\theta$--critical cycle of some $\theta$--critical component $H$ of $G$.
\end{corollary}
\begin{proof} Assume that $C$ is a $\theta$--critical cycle of a $\theta$--critical component $H$ of $G$. Then, by Proposition~\ref{prop:critical_remove}, we have
\[
m_\theta(G \setminus C) - m_\theta(G) = m_\theta(H \setminus C) - m_\theta(H) = -1,
\]  
\noindent so $C$ is a $\theta$--critical cycle of $G$. This proves one direction of the statement.

Now assume that $C$ is a $\theta$--critical cycle of $G$. Then, by Corollary~\ref{cor:critical_cycle}, $C$ is contained in some $\theta$--critical component $H$ of $G$. Once again, by Proposition~\ref{prop:critical_remove}, we have
\[
m_\theta(H\setminus C)-m_\theta(H)=m_\theta(G \setminus C)-m_\theta(G)=-1,
\]  
\noindent so $C$ is a $\theta$--critical cycle of $H$. This completes the proof.
\end{proof}

We now state one final result needed for the proof of Theorem~\ref{thm:main_theorem_technical}. It can be seen as an analogue of Theorem~\ref{thm:stability} for the removal of vertices in $\theta$–critical components, and it also complements Proposition~\ref{prop:critical_remove}.

\begin{proposition}\label{prop:critical_stability} Let $Z\subseteq V(H)$ for a $\theta$--critical component $H$ of $G$. Then $\alpha_i^{G\setminus Z}(\theta)=\alpha_i^{H\setminus Z}(\theta)$ for every $i\in V(H\setminus Z)$.
\end{proposition}
\begin{proof} We proceed by induction on $k\defeq\lvert\partial_G 0^G_\theta\rvert$. Let $i\in V(H\setminus Z)$.

If $k = 0$, then, since $H$ is a component of $G$ and $\alpha_i^{G\setminus Z}(x)$ depends only on the component of $G\setminus Z$ containing $i$, the result follows immediately. Now assume $k\geq 1$ and that the statement holds for all multi-graphs $\tilde{G}$ with $\lvert\partial_{\tilde{G}} 0^{\tilde{G}}_\theta\rvert <k$.

Consider $j\in\partial_G 0^G_\theta$ and $G\setminus j$. By Theorem~\ref{thm:stability}, $H$ is a $\theta$--critical component of $G\setminus j$. Since $\lvert\partial_{G\setminus j} 0^{G\setminus j}_\theta\rvert < k$, by the induction hypothesis, $\alpha_i^{G\setminus (Z\cup\{j\})}(\theta)=\alpha_i^{(G\setminus j)\setminus Z}(\theta) =\alpha_i^{H\setminus Z}(\theta)$.

Moreover, by Theorem~\ref{thm:matched_special} and Proposition~\ref{prop:subset_in_gallai_edmonds}, since $Z\cup\{i\}\subseteq V(H)$ and $H$ is a $\theta$--critical component of $G$, we have $\partial_G 0^G_\theta\subseteq\infty^{G\setminus Z}_\theta$ and $\partial_G 0^G_\theta\subseteq\infty^{G\setminus (Z\cup\{i\})}_\theta$. In particular, $j\in\infty^{G\setminus Z}_\theta$ and $j\in\infty^{G\setminus (Z\cup\{i\})}_\theta$. By Lemma~\ref{lem:contraction} (or~\cite[Prop. 20--24 or Fig. 5]{spier2023refined}), as $j\in\infty^{G\setminus Z}_\theta$ and $j\in\infty^{G\setminus (Z\cup\{i\})}_\theta$, it follows that $\alpha_i^{G\setminus Z}(\theta) =\alpha_i^{G\setminus (Z\cup\{j\})}(\theta)$, and therefore $\alpha_i^{G\setminus Z}(\theta) =\alpha_i^{H\setminus Z}(\theta)$. This completes the induction step and the proof.
\end{proof}

Finally, we note that the results presented in this section are instances of more general statements, to which we intend to return in future work.

%%%%%%%%%%%%%%%%%%%%%%%%%%%%%%%%%%%%%%%%%%%%%%%%%%%%%%%%%%%%%%%%%%%%%%%%%%%%%%%%

\section{Proof of main results}\label{sec:proof_of_main_results}\

In this section, we prove Theorems~\ref{thm:main_theorem_technical} and~\ref{thm:density_of_states_refined}.

We begin with Theorem~\ref{thm:main_theorem_technical}, starting with the case of a connected $\theta$–critical multi-graph $G$. By Proposition~\ref{prop:aomoto_implies_refined_aomoto} and Corollary~\ref{cor:refined_aomoto_in_gallai_edmonds}, such a multi-graph $G$ has a $\theta$--Aomoto subset if and only if it is a $\theta$--critical tree. The following result provides a converse to Corollary~\ref{cor:critical_cycle} for such a multi-graph $G$.

\begin{theorem}\label{thm:theta_critical_cycle} Let $G$ be a connected $\theta$--critical multi-graph that is not a tree. Then $G$ has a $\theta$--critical cycle. 
\end{theorem}
\begin{proof} Note that, by Theorem~\ref{thm:gallais_lemma}, we have $m_\theta(G) = 1$.
Since $G$ is not a tree, it contains at least one cycle $C$, which may consist of a loop or a pair of parallel edges. If $m_\theta(G\setminus C) = 0$, then we are done, so assume that $m_\theta(G\setminus C)\geq 1$. 

We claim that there exists some path $\tilde{P}$ of $G$ with $W_\theta(\tilde{P})=0$ (as defined before Lemma~\ref{lem:entry_of_path}). Let $P$ be any path obtained from $C$ by starting at a vertex $i$ of $C$ and traversing the cycle in one of the two directions. Since $P$ and $C$ contain the same vertices, we have $m_\theta(G\setminus P) = m_\theta(G\setminus C)\geq 1$. By Theorem~\ref{thm:gallais_lemma}, $m_\theta(G) = 1$, and thus by Lemma~\ref{lem:entry_of_path} it follows that $W_\theta(P)\geq 0$. Hence, as $i\in 0^G_\theta$, by Lemma~\ref{lem:entry_of_path_start_critical}, there exists some sub-path $\tilde{P}$ of $P$ with $W_\theta(\tilde{P})=0$, proving our claim.

Let $\hat{P}:\, i_1,\dots, i_k$ be a path in $G$ of minimum length with $W_\theta(\hat{P}) = 0$. Since $i_1\in 0_\theta^G$, it follows that $k\geq 2$ and $i_1\neq i_k$. We claim that $i_k\in\infty^{G\setminus\{i_1,\dots, i_{k-1}\}}$ and $i_1\in\infty^{G\setminus\{i_k,\dots, i_2\}}$. Indeed, by Lemma~\ref{lem:entry_of_path_start_critical}, since $i_1\in 0_\theta^G$ and there is no sub-path $P'$ of $\hat{P}$ with $W_\theta(P') = 0$, it follows that $i_k\in\infty^{G\setminus\{i_1,\dots, i_{k-1}\}}$. Similarly, by considering the reverse path $\hat{P}^{-1}:\, i_k,\dots, i_1$, we conclude that $i_1\in\infty^{G\setminus\{i_k,\dots, i_2\}}$, proving the claim.

Since $i_k\in\infty^{G\setminus\{i_1,\dots, i_{k-1}\}}$ and $i_1\in\infty^{G\setminus\{i_k,\dots, i_2\}}$, Lemma~\ref{lem:alpha_recursion} (see also~\cite[Lem. 15]{spier2023refined} or~\cite[Lem. 3.4]{godsil1995algebraic}) implies that $i_1$ and $i_k$ are adjacent in $G$ to vertices $u$ and $v$ in $0_\theta^{G\setminus\hat{P}}$, respectively, possibly with $u=v$. As $W_\theta(\hat{P})=0$ and $m_\theta(G)=1$, we have, by Lemma~\ref{lem:entry_of_path}, $m_\theta(G\setminus\hat{P})=1$. Thus, since $u, v\in 0^{G\setminus\hat{P}}_\theta$, Lemma~\ref{lem:mult1_implies_critical_path} implies the existence of a $\theta$--critical path $Q: v\to u$ in $G\setminus\hat{P}$. 

Let $\hat{C}$ be the cycle obtained by concatenating $\hat{P}$, the edge $i_k v$, the path $Q$, and the edge $u i_1$. Since $i_1\neq i_k$, $\hat{C}$ is indeed a cycle, even if $u=v$. We then have  
\[
m_\theta(G\setminus\hat{C}) = m_\theta((G\setminus\hat{P})\setminus Q) = m_\theta(G\setminus\hat{P}) - 1 = 0,
\]
\noindent so $\hat{C}$ is a $\theta$--critical cycle of $G$.
\end{proof}

Before moving forward, we mention what Theorem~\ref{thm:theta_critical_cycle} corresponds to in the classical setting, where the multi-graph is a graph, vertex and arc weights are $0$ and $1$, respectively, and $\theta = 0$. In this case, a connected $\theta$--critical graph is exactly a \emph{factor--critical} graph~\cite[p. 89]{lovasz1986matching}. This follows from Gallai’s Lemma~\cite[Thm. 3.1.13]{lovasz1986matching} (see also Theorem~\ref{thm:gallais_lemma}). In a factor-critical graph, the deletion of any vertex leaves a graph that admits a perfect matching.

The only bipartite factor--critical graph is the single-vertex graph $K_1$~\cite[p. 89]{lovasz1986matching} (assuming that the empty graph has a perfect matching). In particular, the only factor--critical tree is $K_1$.

In this classical setting, Theorem~\ref{thm:theta_critical_cycle} states that for every factor-critical graph $G$ that is not $K_1$, there exists a cycle $C$ such that $G\setminus C$ has a perfect matching. In this case, the following stronger statement is true.

\begin{lemma}\label{lemma:factor_critical_case} Assume that $G$ is factor-critical and let $i$ and $j$ be neighbors in $G$. Then, there exists an odd cycle $C$ containing $i$ and $j$ such that $G\setminus C$ has a perfect matching.
\end{lemma}
\begin{proof} The proof is similar to that of~\cite[Thm. 5.5.1]{lovasz1986matching}. Since $G$ is factor-critical, there exist perfect matchings $M_i$ and $M_j$ of $G\setminus i$ and $G\setminus j$, respectively. There must be an $M_i$--$M_j$ alternating path $P$ connecting $i$ and $j$, and its length is necessarily even. Let $C$ be the cycle formed by $P$ followed by the edge $ji$. Then, $C$ is odd and $G\setminus C$ has a perfect matching. 
\end{proof}

The following result provides the main mechanism for the proof of Theorem~\ref{thm:main_theorem_technical}.

\begin{proposition}\label{prop:removing_critical_cycle} If $G$ has no $\theta$--Aomoto subset and $C$ is a $\theta$--critical cycle of $G$, then $G\setminus C$ has no $\theta$--Aomoto subset.
\end{proposition}
\begin{proof} Assume, by contradiction, that $G\setminus C$ has a $\theta$--Aomoto subset. By Proposition~\ref{prop:aomoto_implies_refined_aomoto}, $G\setminus C$ then has a refined $\theta$--Aomoto subset $S$. If $C$ is not adjacent to $S$ in $G$, then $S$ would also be a refined $\theta$--Aomoto subset of $G$, which is impossible. Hence, $C$ must be adjacent to $S$ in $G$.  

By Corollary~\ref{cor:critical_in_comp_equiv_critical}, $C$ is contained in a $\theta$--critical component $H$ of $G$. Thus at least one vertex of $S$ lies in $V(H \setminus C) \cup \partial_G H$. By Proposition~\ref{prop:subset_in_gallai_edmonds}, we know that $\partial_G H\subseteq\partial_G 0^G_\theta\subseteq\infty^{G\setminus C}_\theta$. But by Corollary~\ref{cor:refined_aomoto_in_gallai_edmonds}, we have $S\subseteq 0^{G\setminus C}_\theta$, so $S\cap\partial_G H =\emptyset$. It follows that there exists a vertex $i\in S\cap V(H\setminus C)$. 

Since $i\in S\subseteq 0^{G\setminus C}_\theta$, we have $\alpha_i^{G\setminus C}(\theta) = 0$. By Proposition~\ref{prop:critical_stability}, it follows that $\alpha_i^{H\setminus C}(\theta) =\alpha_i^{G\setminus C}(\theta)=0$, that is, $i\in 0^{H\setminus C}_\theta$. This is a contradiction, since Corollary~\ref{cor:critical_in_comp_equiv_critical} implies that $C$ is also a $\theta$--critical cycle of $H$, and hence $m_\theta(H\setminus C)=0$.
\end{proof}

Now we are ready for the proof of our main result.

\begin{proof}[Proof of Theorem~\ref{thm:main_theorem}] Let $k\defeq m_\theta(G)$. If $k=0$, there is nothing to prove. Therefore, assume $k\geq 1$ and, by induction hypothesis, that the statement holds for all multi-graphs $\tilde{G}$ with $m_\theta(\tilde{G})=k-1$.

As $G$ has no $\theta$--Aomoto subset, it follows that at least one $\theta$--critical component $H$ of $G$ is not a tree. Indeed, if all $\theta$--critical components of $G$ were trees, then, by Corollary~\ref{cor:refined_aomoto_in_gallai_edmonds} and Theorem~\ref{thm:matched_special}, $0^G_\theta$ would be a refined $\theta$--Aomoto subset. By Theorem~\ref{thm:theta_critical_cycle}, $H$ contains a $\theta$--critical cycle $C$. By Corollary~\ref{cor:critical_in_comp_equiv_critical}, $C$ is also a $\theta$--critical cycle of $G$. Hence $m_\theta(G\setminus C)=m_\theta(G)-1$. 

By Proposition~\ref{prop:removing_critical_cycle}, $G\setminus C$ has no $\theta$--Aomoto subset. Since $m_\theta(G\setminus C) = k-1$, the induction hypothesis guarantees the existence of disjoint cycles $C_1,\dots, C_{k-1}$ such that $m_\theta((G\setminus C)\setminus (C_1\sqcup\cdots\sqcup C_{k-1})) = 0$. If we define $C_k\defeq C$, then $m_\theta(G\setminus (C_1\sqcup\cdots\sqcup C_k)) = 0$, completing the proof.
\end{proof}  

In the same way that Lemma~\ref{lemma:factor_critical_case} provides a simple proof of Theorem~\ref{thm:theta_critical_cycle} in the classical case, it is also possible to obtain a more direct proof of Theorem~\ref{thm:main_theorem_technical} in this setting. This can be done by using the classical Gallai--Edmonds decomposition and the theory developed in Lovász and Plummer's book~\cite{lovasz1986matching}.

We now proceed to the proof of Theorem~\ref{thm:density_of_states_refined}. We start with the following result, which will allow us to define the maximal refined $\theta$--Aomoto subset of a multi-graph.

\begin{proposition}\label{prop:refined_aomoto_union} Let $S$ and $S'$ be refined $\theta$--Aomoto subsets of $G$. Then $S\cup S'$ is a refined $\theta$--Aomoto subset of $G$.
\end{proposition}
\begin{proof} By Corollary~\ref{cor:refined_aomoto_in_gallai_edmonds}, both $S$ and $S'$ are disjoint unions of $\theta$--critical components of $G$ that are trees. As a consequence, $S \cup S'$ is also a disjoint union of $\theta$--critical components of $G$ that are trees. Also note that $\partial_G (S\cup S')=\partial_G S\cup \partial_G S'$. 

Now, consider $\emptyset \neq U \subseteq \partial_G (S \cup S')$. If $U \subseteq \partial_G S$ or $U \subseteq \partial_G S'$, then $U$ is adjacent to at least $\lvert U \rvert + 1$ components of $G[S \cup S']$.  Thus, assume instead that $\emptyset \neq U \cap \partial_G S \subseteq \partial_G S$ and $\emptyset \neq U \setminus \partial_G S \subseteq \partial_G S'$. Then, $U \cap \partial_G S$ is adjacent to at least $\lvert U \cap \partial_G S \rvert + 1$ components of $G[S]$, and $U \setminus \partial_G S$ is adjacent to at least $\lvert U \setminus \partial_G S \rvert + 1$ components of $G[S']$, none of which are components of $G[S]$. Hence, $U$ is adjacent to at least  
\[
(\lvert U \cap \partial_G S \rvert + 1) + (\lvert U \setminus \partial_G S \rvert + 1) = \lvert U \rvert + 2
\]  
\noindent components of $G[S \cup S']$.  

This shows that $S \cup S'$ is a refined $\theta$--Aomoto subset.  
\end{proof}

The next result explains why a maximal refined $\theta$--Aomoto subset should be the unique maximizer in Theorem~\ref{thm:density_of_states_refined}.

\begin{proposition}\label{prop:refined_aomoto_order} Let $S$ and $S'$ be refined $\theta$--Aomoto subsets of $G$ with $S \subsetneq S'$. Then  
\[
\mathrm{cc}(G[S]) - \lvert \partial_G S \rvert < \mathrm{cc}(G[S']) - \lvert \partial_G S' \rvert.
\] 
\end{proposition}
\begin{proof} Note that, by Corollary~\ref{cor:refined_aomoto_in_gallai_edmonds}, all components of $G[S]$ are components of $G[S']$, $\mathrm{cc}(G[S]) < \mathrm{cc}(G[S'])$ and $\partial_G S \subseteq \partial_G S'$. If $\partial_G S = \partial_G S'$, then the result immediately follows. Therefore, assume that $\partial_G S' \setminus \partial_G S \neq \emptyset$. Then, $\partial_G S' \setminus \partial_G S$ is adjacent to at least $\lvert \partial_G S' \setminus \partial_G S \rvert + 1$ components of $G[S']$, none of which are components of $G[S]$. Hence,  
\[
\mathrm{cc}(G[S']) \geq \mathrm{cc}(G[S]) + \lvert \partial_G S' \setminus \partial_G S \rvert + 1,
\]
\noindent which, upon rearranging, proves the statement.
\end{proof}

Now we are ready for the proof of Theorem~\ref{thm:density_of_states_refined}.

\begin{proof}[Proof of Theorem~\ref{thm:density_of_states_refined}] Note that, by Proposition~\ref{prop:aomoto_implies_refined_aomoto}, the maximization problem in Theorem~\ref{thm:density_of_states} has a maximizer that is a refined $\theta$--Aomoto subset. Therefore,  
\begin{equation}\label{eq:refined_maximization}
\tau(\theta)=\max_{S\in\mathcal{R}_\theta^G} \frac{\mathrm{cc}(G[S]) -\lvert\partial_G S \rvert}{\lvert V(G) \rvert}.
\end{equation}
Let $\tilde{S}$ denote the union of all refined $\theta$--Aomoto subsets of $G$. By Proposition~\ref{prop:refined_aomoto_union}, $\tilde{S}$ is the maximal refined $\theta$--Aomoto subset of $G$. Moreover, by Proposition~\ref{prop:refined_aomoto_order}, we have  
\[
\mathrm{cc}(G[S]) - \lvert \partial_G S \rvert < \mathrm{cc}(G[\tilde{S}]) - \lvert \partial_G \tilde{S} \rvert,
\]  
\noindent for every refined $\theta$--Aomoto subset $S \subsetneq \tilde{S}$. Hence, $\tilde{S}$ is the unique maximizer in the maximization problem stated in Equation~\ref{eq:refined_maximization}.
\end{proof}

%%%%%%%%%%%%%%%%%%%%%%%%%%%%%%%%%%%%%%%%%%%%%%%%%%%%%%%%%%%%%%%%%%%%%%%%%%%%%%%%

\section*{Acknowledgements}

The author is grateful to Wendo Li, Michael Magee, Mostafa Sabri, and Joe Thomas for their comments on a previous version of this article. The author also acknowledges fruitful conversations on the topic of this paper with Chris Godsil. Some of the results presented here were obtained while the author was a Ph.D. student at IMPA, Brazil.

%%%%%%%%%%%%%%%%%%%%%%%%%%%%%%%%%%%%%%%%%%%%%%%%%%%%%%%%%%%%%%%%%%%%%%%%%%%%%%%%
\bibliographystyle{plain}
\IfFileExists{references.bib}
{\bibliography{references.bib}}
{\bibliography{../references}}

\begin{thebibliography}{10}

\bibitem{aomoto1988algebraic}
K.~Aomoto.
\newblock Algebraic equations for green kernel on a tree.
\newblock {\em Proceedings of the Japan Academy, Series A, Mathematical Sciences}, 64:123--125, 1988.

\bibitem{aomoto1991point}
K.~Aomoto.
\newblock Point spectrum on a quasihomogeneous tree.
\newblock {\em Pacific J. Math.}, 147:231--242, 1991.

\bibitem{arizmendi2024universality}
Octavio Arizmendi, Guillaume C{\'e}bron, Roland Speicher, and Sheng Yin.
\newblock Universality of free random variables: atoms for non-commutative rational functions.
\newblock {\em Advances in Mathematics}, 443:109595, 2024.

\bibitem{avni2020periodic}
N.~Avni, J.~Breuer, and B.~Simon.
\newblock Periodic {J}acobi matrices on trees.
\newblock {\em Adv. Math.}, 370:107241, 2020.

\bibitem{avni2022periodic}
Nir Avni, Jonathan Breuer, Gil Kalai, and Barry Simon.
\newblock Periodic boundary conditions for periodic {J}acobi matrices on trees.
\newblock {\em Pure and Applied Functional Analysis}, 7(2):489--502, 2022.

\bibitem{banks2024useful}
Jess Banks, Jonathan Breuer, Jorge Garza-Vargas, Eyal Seelig, and Barry Simon.
\newblock A useful formula for periodic {J}acobi matrices on trees.
\newblock {\em Proceedings of the National Academy of Sciences}, 121(23):e2315218121, 2024.

\bibitem{banks2022point}
Jess Banks, Jorge Garza-Vargas, and Satyaki Mukherjee.
\newblock Point spectrum of periodic operators on universal covering trees.
\newblock {\em International Mathematics Research Notices}, 2022(22):17713--17744, 2022.

\bibitem{bencs2022atoms}
Ferenc Bencs and Andr{\'a}s M{\'e}sz{\'a}ros.
\newblock Atoms of the matching measure.
\newblock {\em Electronic Journal of Probability}, 27:1--38, 2022.

\bibitem{bordenave2019eigenvalues}
C.~Bordenave and B.~Collins.
\newblock Eigenvalues of random lifts and polynomials of random permutation matrices.
\newblock {\em Ann. of Math.(2)}, 190:811--875, 2019.

\bibitem{christiansen2021remarks}
J.~S. Christiansen, B.~Simon, and M.~Zinchenko.
\newblock Remarks on periodic {J}acobi matrices on trees.
\newblock {\em J. Math. Phys.}, 62:042101, 2021.

\bibitem{coutinho2023strong}
Gabriel Coutinho, Emanuel Juliano, and Thom{\'a}s~Jung Spier.
\newblock Strong cospectrality in trees.
\newblock {\em Algebraic Combinatorics}, 6(4):955--963, 2023.

\bibitem{edmonds1965paths}
J.~Edmonds.
\newblock Paths, trees, and flowers.
\newblock {\em Canad. J. Math.}, 17:449--467, 1965.

\bibitem{figa1994harmonic}
Alessandro Figa-Talamanca and Tim Steger.
\newblock Harmonic analysis for anisotropic random walks on homogeneous trees, 1994.
\newblock {\em Memoirs of the American Mathematical Society}, 110(531):531, 1994.

\bibitem{gallai1963kritische}
T.~Gallai.
\newblock Kritische graphen {II}.
\newblock {\em Magyar Tud. Akad. Mat. Kutato Int. Kozl.}, 8:373--395, 1963.

\bibitem{garza2023spectra}
Jorge Garza-Vargas and Archit Kulkarni.
\newblock Spectra of infinite graphs via freeness with amalgamation.
\newblock {\em Canadian Journal of Mathematics}, 75(5):1633--1684, 2023.

\bibitem{godsil_matchings}
C.~D. Godsil.
\newblock Matchings and walks in graphs.
\newblock {\em J. Graph Theory}, 5:285--297, 1981.

\bibitem{godsil1995algebraic}
C.~D. Godsil.
\newblock Algebraic matching theory.
\newblock {\em Electron. J. Combin.}, 2:R8, 1995.

\bibitem{godsil1978matching}
C.~D. Godsil and I.~Gutman.
\newblock On the matching polynomial of a graph.
\newblock {\em Algebraic Methods in Graph Theory}, pages 241--249, 1981.

\bibitem{gross2001topological}
Jonathan~L Gross and Thomas~W Tucker.
\newblock {\em Topological graph theory}.
\newblock Courier Corporation, 2001.

\bibitem{gutman1986some}
I.~Gutman.
\newblock Some relations for graphic polynomials.
\newblock {\em Publications de L'Institut Mathematique, Nouvelle Serie Tome}, 39 (53):55--62, 1986.

\bibitem{gutman1981cyclic}
I.~Gutman and O.~E. Polansky.
\newblock Cyclic conjugation and the {H}\"uckel molecular orbital model.
\newblock {\em Theoret. Chim. Acta}, 60:203--226, 1981.

\bibitem{harary1962determinant}
Frank Harary.
\newblock The determinant of the adjacency matrix of a graph.
\newblock {\em Siam Review}, 4(3):202--210, 1962.

\bibitem{heilmann-lieb}
O.~J. Heilmann and E.~H. Lieb.
\newblock Theory of monomer-dimer systems.
\newblock {\em Comm. Math. Phys.}, 25:190--232, 1972.

\bibitem{higuchi2009spectral}
Yusuke Higuchi and Yuji Nomura.
\newblock Spectral structure of the laplacian on a covering graph.
\newblock {\em European Journal of Combinatorics}, 30(2):570--585, 2009.

\bibitem{johnson2018eigenvalues}
C.~R. Johnson and C.~M. Saiago.
\newblock {\em Eigenvalues, multiplicities and graphs}.
\newblock Cambridge University Press, (2018).

\bibitem{ku2010analogue}
C.~Y. Ku and W.~Chen.
\newblock An analogue of the {G}allai--{E}dmonds structure theorem for non-zero roots of the matching polynomial.
\newblock {\em J. Combin. Theory Ser. B}, 100:119--127, 2010.

\bibitem{ku2009maximum}
C.~Y. Ku and K.~B. Wong.
\newblock Maximum multiplicity of a root of the matching polynomial of a tree and minimum path cover.
\newblock {\em Electron. J. Combin.}, 16:R81, 2009.

\bibitem{ku2010extensions}
C.~Y. Ku and K.~B. Wong.
\newblock Extensions of barrier sets to nonzero roots of the matching polynomial.
\newblock {\em Discrete Math.}, 310:3544--3550, 2010.

\bibitem{ku2013gallai}
C.~Y. Ku and K.~B. Wong.
\newblock Gallai--{E}dmonds structure theorem for weighted matching polynomial.
\newblock {\em Linear Algebra Appl.}, 439:3387--3411, 2013.

\bibitem{li2025eigenvalues}
Wenbo Li, Michael Magee, Mostafa Sabri, and Joe Thomas.
\newblock Eigenvalues of maximal abelian covers.
\newblock {\em arXiv preprint arXiv:2508.17332}, 2025.

\bibitem{lovasz1986matching}
L.~Lov{\'a}sz and M.~Plummer.
\newblock Matching theory.
\newblock {\em North-{H}olland mathematics studies}, 121, 1986.

\bibitem{interlacing_familiesI}
A.~Marcus, D.~A. Spielman, and N.~Srivastava.
\newblock Interlacing families {I}: Bipartite {R}amanujan graphs of all degrees.
\newblock {\em Ann. of Math.(2)}, 182:307--325, 2015.

\bibitem{sabri2023flat}
Mostafa Sabri and Pierre Youssef.
\newblock Flat bands of periodic graphs.
\newblock {\em Journal of Mathematical Physics}, 64(9), 2023.

\bibitem{spier2021graph}
Thom{\'a}s~Jung Spier.
\newblock {\em Graph Continued Fractions}.
\newblock PhD thesis, 2021.

\bibitem{spier2023refined}
Thom{\'a}s~Jung Spier.
\newblock A refined {G}allai-{E}dmonds structure theorem for weighted matching polynomials.
\newblock {\em Discrete Mathematics}, 346(3):113244, 2023.

\bibitem{viennot1985combinatorial}
G.~Viennot.
\newblock A combinatorial theory for general orthogonal polynomials with extensions and applications.
\newblock {\em Polyn{\^o}mes Orthogonaux et Applications}, pages 139--157, 1985.

\end{thebibliography}

%%%%%%%%%%%%%%%%%%%%%%%%%%%%%%%%%%%%%%%%%%%%%%%%%%%%%%%%%%%%%%%%%%%%%%%%%%%%%%%%
	
\end{document}